\newif\ifpictures
\picturestrue

\newif\ifcomment
\commenttrue

\documentclass[12pt]{amsart}
\usepackage{latex_base}
\usepackage{macros}
\usepackage{multirow}
\usepackage{xifthen} %Ifelse statements for tikz.
\usetikzlibrary{shapes,snakes}
\newtest{\mycondition}[2]{%
	\cnttest{(#1)+(#2)}<{5}%
	}

\newcommand{\Rn}{\mathbb{R}^n}
\newcommand{\Hratio}{\mathcal{H}}
\DeclareMathOperator{\ListHead}{head}
\DeclareMathOperator{\ListTail}{tail}

\DeclareMathOperator{\MidD}{Mid}
\DeclareMathOperator{\GL}{GL}
\DeclareMathOperator{\rank}{rank}
\DeclareMathOperator{\coeff}{coeff}
\DeclareMathOperator{\MMSPreservingGroup}{\mathcal{M}_n}

\author{Jacob Hartzer}

\address{Jacob Hartzer, Texas A\&M University, Mechanical Engineering Department, College Station, TX 77843, USA\medskip}

\email{jmhartzer@tamu.edu}

\author{Olivia R\"ohrig}

\address{Olivia R\"ohrig, Technische Universit\"at Berlin, Institut f\"ur Mathematik, Stra{\ss}e des 17.~Juni 136, 10623 Berlin,
 Germany\medskip}

\email{roehrig@math.tu-berlin.de}

\author{Timo de Wolff}

\address{Timo de Wolff, Technische Universit\"at Braunschweig, Institut f\"ur Analysis und Algebra, AG Algebra, Universit\"atsplatz 2, 38106 Braunschweig,
 Germany\medskip}

\email{t.de-wolff@tu-braunschweig.de}

\author{O\u{g}uzhan Y{\"{u}}r{\"{u}}k}

\address{Oguzhan Y{\"{u}}r{\"{u}}k, Technische Universit\"at Braunschweig, Institut f\"ur Analysis und Algebra, AG Algebra, Universit\"atsplatz 2, 38106 Braunschweig, Germany\medskip}

\email{oguyueru@tu-braunschweig.de}


\title{Initial Steps in the Classification of Maximal Mediated Sets}

\date{\today}

\begin{document}

\begin{abstract}
Maximal mediated sets (MMS), introduced by Reznick, are distinguished subsets of lattice points in integral polytopes with even vertices. MMS of Newton polytopes of AGI-forms and nonnegative circuit polynomials determine whether these polynomials are sums of squares.

In this article, we take initial steps in classifying MMS both theoretically and practically. Theoretically, we show that MMS of simplices are isomorphic if and only if the simplices generate the same lattice up to permutations. Furthermore, we generalize a result of Iliman and the third author. Practically, we fully characterize the MMS for all simplices of sufficiently small dimensions and maximal 1-norms. In particular, we experimentally prove a conjecture by Reznick for 2 dimensional simplices up to maximal 1-norm 150 and provide indications on the distribution of the density of MMS.
\end{abstract}

\maketitle
\section{Introduction}

Studying \textit{\struc{nonnegativity}} of real, multivariate polynomials is a key problem in real algebraic geometry.
Since deciding nonnegativity is well-known to be a hard problem both in theory and practice, one is interested in \textit{\struc{certificates}} of nonnegativity, i.e., a special structure of real polynomials, which implies nonnegativity and is easier to test. The classical as well as most common certificate of nonnegativity are \struc{\textit{sums of squares (SOS)}}; see e.g., \cite{Marshall:Book}, and see \cite{Reznick:SurveyHilbert17th} for a historical overview.

Let $\struc{\R[\Vector{x}]} = \R[x_1,\ldots,x_n]$ denote the ring of real $n$-variate polynomials and let $\struc{\R[\Vector{x}]_{2d}}$ denote the vector space of polynomials in $\R[\Vector{x}]$ of degree at most $2d$. For arbitrary fixed $n$ and $d$  one considers in $\R[\Vector{x}]_{2d}$ the full-dimensional convex cones
\begin{align*}
    \struc{P_{n,2d}} \ := \ \{f \in \R[\Vector{x}]_{n,2d} \ : \ f(\Vector{x}) \geq 0 \text{ for all } \Vector{x} \in \R^n\}, \quad \text{ and }
\end{align*}
\begin{align*}
    \struc{\Sigma_{n,2d}} \ := \ \left\{f \in P_{n,2d} \ : \ f = \sum_{j=1}^{r} s_j^2 \text{ with } s_1,\ldots,s_r \in \R[\Vector{x}]_{n,d}\right\}.
\end{align*}

In 1888, Hilbert \cite{Hilbert:1888} showed that $P_{n,2d} \neq \Sigma_{n,2d}$ unless $(n,2d) \in \{(k,2),(2,4),(1,k) \ : \ k \in \N_{> 0}\}$. Since then nonnegativity of real polynomials and sums of squares has been an active field of research. This holds especially for the last two decades due to their vast applications in polynomial optimization; see e.g., \cite{Blekherman:Parrilo:Thomas:SDOptAndConvAlgGeo}, \cite{Laserre:MomentsPosPoly}.
%Laserre book Ab Intro. to Poly. Opt and Semi-Algebraic Opt.

Another classical way to certify nonnegativity of a polynomial is the \struc{\textit{inequality of arithmetic and geometric means}}, short \struc{\textit{AM-GM inequality}}, which states that:
\begin{align}
	\sum_{j = 0}^d \lambda_j t_j - \prod_{j = 0}^d t_j^{\lambda_j} \ \geq \ 0, \label{Equ:AMGMInequality}
\end{align}
if $t_j \geq 0$, $\lambda_j \geq 0$ and $\sum_{j=1}^d \lambda_j = 1$.

In 1891, Hurwitz provided a new proof for the AM-GM inequality \cite{Hurwitz:AMGM}, in particular proving nonnegativity of polynomials of the form 
\begin{align}
	\sum_{j = 1}^{2n} x_j^{2n} - 2n \prod_{j=1}^{2n}x_i. \label{Equ:HurwitzPolynomial}
\end{align}
Furthermore, he showed that polynomials of the form \cref{Equ:HurwitzPolynomial} are in $\Sigma_{n,2d}$. 
Following Reznick's notation, we refer these polynomials as \textit{\struc{Hurwitz forms}}.

In 1989, Reznick \cite{Reznick:AGI} vastly generalized this approach, building on Hurwitz's \cite{Hurwitz:AMGM} work and earlier results by Choi and Lam, e.g., \cite{Choi:Lam:ExtremalPositiveSemidefiniteForms}. Using the AM-GM inequality, he generated a class of nonnegative polynomials called \textit{\struc{AGI-forms}}. 
In recent years, Iliman and the third author generalized simplicial AGI-forms to circuit polynomials \cite{Iliman:deWolff:Circuits}, and connected these to polynomial optimization; see e.g., \cite{Iliman:deWolff:LowerBoundsforSimplexNewtPoly}.
Moreover, nonnegative functions recently introduced by Chandrasekaran and Shah \cite{Chandrasekaran:Shah:SAGE}, which are motivated by signomial programming, can be seen as a generalization of AGI-forms; see also \cite{DeWolff:Forsgard:AlgebraicBoundarySONC}.

In particular, a nonnegative AGI-form or a nonnegative circuit polynomial is not SOS in general; see \cite[Corollary 4.9]{Reznick:AGI} and \cite[Theorem 5.2]{Iliman:deWolff:Circuits}. 
The easiest example is the \textit{\struc{Motzkin polynomial}}:
\begin{align}
	\label{Equ:MotzkinPolynomial}
	\struc{M(x_1,x_2)} \ := \ 1 + x_1^2x_2^4 + x_1^4x_2^2 - 3 x_1^2x_2^2.
\end{align}
It is well-known as the first nonnegative polynomial in the literature, which is not SOS; see \cite{Motzkin:AMGMIneq} and \cite{Reznick:SurveyHilbert17th}.

\medskip

The \textit{\struc{maximal mediated set}} of an integral simplex $S$ with vertices $\vertices{S}$ in $(2\Z)^n$ is the largest subset $M$ of lattice points in $\Z^n \cap S$ satisfying the following two properties:
\begin{enumerate}
	\item $\vertices{S} \subset M$, and
	\item if $p \in M$, then there exist $q_1,q_2 \in (2\Z)^n \cap M$ with $p = \dfrac{1}{2}(q_1 + q_2)$. 
\end{enumerate}
For further details see \cref{Definition:MediatedSet} and \cref{Definitiom:MMS}.

\medskip

In \cite[Corollary 4.9]{Reznick:AGI} Reznick proved that a nonnegative simplicial AGI-form is a sum of squares if and only if the exponent of one distinguished term belongs to the \textit{maximal mediated set} of its Newton polytope. Iliman and the third author generalized this statement to circuit polynomials \cite{Iliman:deWolff:Circuits}. More specifically, \cite[Theorem 5.2]{Iliman:deWolff:Circuits} states that a nonnegative circuit polynomial $f$ is a sum of squares if and only if the support of $f$ is contained in the maximal mediated set of the support of $f$. Thus, the question of whether a nonnegative circuit polynomial is a sum of squares depends on the support \textit{alone} and hence is \textit{purely combinatorial}. This is in sharp contrast to the case of general nonnegative polynomials, where this question also depends on the coefficients.

\medskip

In this article we initiate a systematic study of maximal mediated sets both from the theoretical and the practical point of view. On the theoretical side, we re-examine an observation by Reznick \cite[Page 445]{Reznick:AGI} in detail for integral simplices.
First, we introduce the \struc{$h$-\textit{ratio}} of $\vertices{S}$ to measure the density of the maximal mediated set of $S$ in $\Z^n \cap S$; see \cref{Definition:hratio}.
We show that a map $T$ from $\R^n$ to $\R^n$ preserves the $h$-ratio if and only if $T \in (2\Z)^n \rtimes \GL(\Z^n)$. 
Most importantly, we show that for an integral simplex $S$ with even lattice points the associated $h$-ratio is an invariant of an underlying lattice described in \cref{Corollary:MMSisInvariantofLattice}.
Lastly, in \cref{Theorem:SONCwithSimplexNewt}, we generalize the \cite[Theorem 5.2]{Iliman:deWolff:Circuits} and relate maximal mediated sets with a larger class of polynomials.

The majority of our contribution is on the practical side. Here, we perform a large scale computation via \textsc{Polymake} \cite{Ewgeniy:Joswig:polymake:2000} generating a database of MMS.
We compute all maximal mediated sets of dimension $n$ and maximal 1-norm $2d$ for the values of $n$ and $2d$ given in \cref{table:datainfo_full}.
In addition, we investigate the statistics of $h$-ratio by sampling for the cases given in \cref{table:datainfo_sampled}.
We utilize \cref{Corollary:MMSisInvariantofLattice} in order to reduce the size of the database. 

Using this database of maximal mediated sets, we provide indications on the distributions of $h$-ratio over simplicial sets and lattices.
We experimentally show, up to $2d=150$, a claim by Reznick from 1989 about maximal mediated sets for $n=2$; see \cref{Conjecture:ReznickDim2}.
%We remark that one day before submitting our article to the ArXiv, Reznick and Powers published a full proof of this conjecture on the ArXiv; see \cite{Reznick:Powers:NoteOnMMS}.

\subsection*{Acknowledgments}

We thank Bruce Reznick and Sadik Iliman for their helpful comments. We thank the \textsc{Polymake} team, in particular Benjamin Lorenz, for their invaluable technical support.

TdW and OY are supported by the DFG grant WO 2206/1-1.

\section{Preliminaries}
\label{Section:Preliminaries}
We denote the set of nonnegative integers $\{0,1,\dots,n\}$ by $\struc{[n]}$ and vectors by bold characters. 
Given a finite set $L \subseteq \N^n$, \struc{$\#L$} denotes the cardinality of $L$ and $\struc{[L]}$ denotes the list formed by $L$ with lexicographical order.
We denote the convex hull of $L$ by \struc{$\conv(L)$} and the vertices of $\conv(L)$ by \struc{$\vertices{L}$}. 
If $f \in \R[\Vector{x}]$, then we denote by $\struc{\newton{f}}$ the \textit{\struc{Newton polytope}} of $f$, i.e., the convex hull of the exponents of $f$.
A point $\Vector{\alpha} \in \N^n$ is called \textit{\struc{even}} if all of its entries are even. 
A subset of $\N^n$ is an \textit{\struc{even set}} if it consists of even points only.
An even set $\Delta \subseteq (2\Z)^n$ is \textit{\struc{$k$-simplicial}} if $\conv(\Delta) \subset \R^n$ is a $k$-simplex; we write $\struc{\dim(\Delta)} = k$. 
The \struc{\textit{maximal degree}} of a $k$-simplicial set $\Delta$ is defined as the maximum 1-norm of its elements. 
Furthermore, following the notation in \cite{Reznick:AGI}, we call a $k$-simplicial set $\Delta$ a \textit{\struc{trellis}}, if all of its elements have the same 1-norm. Let $\Delta = \left\{ \Vector{v_0}, \dots, \Vector{v_n} \right\} \subset \R^n$ be a lexicographically ordered $k$-simplicial set. Then we denote by $\struc{M_{\Delta}}$ the column matrix of the elements in $\Delta$, i.e.,
\begin{align*}
\struc{M_{\Delta}} \ = \ \begin{bmatrix} 
	\Vector{v_0} & \cdots & \Vector{v_n}
\end{bmatrix}.
\end{align*}

If $\Vector{v_0} = \Vector{0}$ and $M_{\Delta} \in \Z^{n \times (n+1)}$, then we define $\struc{L_{\Delta}} \subset \Z^n$ as the \textit{\struc{lattice}} generated by rows of the matrix obtained by deleting the first column of $M_{\Delta}$.

Given $L\subseteq \Z^n$,  we define a set of \textit{\struc{midpoints}} of $L$ as follows:
\begin{align*}
\struc{\MidD(L)} \ = \ \left\{ \frac{\Vector{s}+\Vector{t}}{2} : \Vector{s},\Vector{t} \in L\cap (2\Z)^n, \Vector{s} \neq \Vector{t}  \right\}.
\end{align*}

\begin{definition}
	\label{Definition:MediatedSet}
	Let $\Delta \subseteq (2\Z)^n$ be a finite even set. Then $L \subseteq \Z^n$ is called \textit{\struc{$\Delta$-mediated}} if
	\begin{align*}
	\Delta \ \subseteq \ L \ \subseteq \ \MidD(L) \cup \Delta.
	\end{align*}
In other words, $L\subseteq \Z^n$ is called $\Delta$-mediated if it contains $\Delta$, and each element of $L \setminus \Delta$ is the midpoint of two distinct even points in $L$.
\end{definition}

In \cite{Reznick:AGI}, \cref{Definition:MediatedSet} is formulated in the context of trellises, however, Reznick's construction works for any set of even lattice points as it was pointed out in \cite{Iliman:deWolff:Circuits}.   

\begin{example}
\label{Example:MediatedSet}
Let $\Delta = \{ (0,0),(4,0),(0,4) \}$, $L_1 = \{ (0,0),(4,0),(0,4),(1,1) \}$ and $L_2 = \{ (0,0),(4,0),(0,4),(2,3) \}$.
$L_1$ is $\Delta$-mediated because $(1,1)$ is the midpoint of $(0,0),(2,2) \in \MidD(L)$. 
However, $L_2$ is not $\Delta$-mediated because $(2,3)$ cannot be written as a midpoint of two even points from $\conv(L_2)$ since it is a vertex of $\conv(L_2)$.
\end{example}

We observe from \cref{Example:MediatedSet} that if $L$ is $\Delta$-mediated, then $L \subseteq \conv(\Delta) \cap \Z^n$.

\begin{definition}
\label{Definitiom:MMS}
Given a finite $\Delta \subset (2\Z)^n$, the \textit{\struc{maximal $\Delta$-mediated set (MMS)}} \textit{\struc{$\Delta^*$}} is the $\Delta$-mediated set that contains every $\Delta$-mediated set.
\end{definition}

MMS are well-defined due to \cref{Theorem:Reznick:MMSExists}. A constructive proof of this theorem is provided in \cite{Reznick:AGI} using \cref{algorithm:MMS1}.

\begin{theorem}[\cite{Reznick:AGI}, Theorem 2.2]
	\label{Theorem:Reznick:MMSExists}
	Given a finite $\Delta \subset (2\Z)^n$, there exists a unique maximal mediated set $\Delta^*$ satisfying
	\begin{align*}
	\Delta \cup \MidD(\Delta) \ \subseteq \ \Delta^* \ \subseteq \ \conv(\Delta) \cap \Z^n.
	\end{align*}
\end{theorem}

\begin{algorithm}[\cite{Reznick:AGI}]
	Given a finite $\Delta \subseteq (2\Z)^n$, the following algorithm computes a non-increasing sequence of subsets that stabilizes at $\Delta^*$. \\
	\INPUT{$\Delta$: finite set of points in $(2\Z)^n$}\\
	\OUTPUT{$\Delta^*$: the $\Delta$-mediated subset of $\Z^n$ that contains every $\Delta$-mediated set}\\
	\begin{algorithmic}[1]
		\State $\Delta^0\gets \conv(\Delta)\cap \Z^n$
		\Repeat
		\State $\Delta^i \gets \MidD(\Delta^{i-1})\cup \Delta$
		\Until{$\Delta^i=\Delta^{i-1}$}
		\State $\Delta^* \gets \Delta^i$
	\end{algorithmic}
	\label{algorithm:MMS1}
\end{algorithm}

Maximal mediated sets arise naturally from the study of nonnegative polynomials supported on a circuit; see \cite{Iliman:deWolff:Circuits}.
\begin{definition}
	\label{Definition:CircuitPolynomial}
	A polynomial $p \in \R[x_1,\dots,x_n]$ is called a \struc{\textit{circuit polynomial}} if it is of the form
	\begin{align*}
	p(\Vector{x}): = c_{\Vector{\beta}} \Vector{x}^{\Vector{\beta}} + \sum_{j=0}^{r} c_{\Vector{\alpha(j)}} \Vector{x}^{\Vector{\alpha(j)}}
	\end{align*} 
	with $r \leq n$, exponents $\Vector{\alpha(j)}, \Vector{\beta} \in \N^n$, and the coefficients $c_{\Vector{\alpha(j)}} \in \R_{>0}$, $c_{\Vector{\beta}} \in \R$ such that $\newton{p}$ is a simplex with even vertices $\Vector{\alpha(0)}, \dots, \Vector{\alpha(r)}$ and the exponent $\Vector{\beta}$ is in the strict interior of $\newton{p}$. \struc{\textit{The MMS associated to} $p$} is the maximal $\Delta$-mediated set with $\Delta = \vertices{\newton{p}}$, denoted by $\struc{\Delta(p)^*}$. 
\end{definition}
Let $p$ be a circuit polynomial given as in \cref{Definition:CircuitPolynomial} such that $p \in P_{n,2d}$. 
In \cite[Theorem 5.2]{Iliman:deWolff:Circuits} authors point out that, such $p$ is in $\Sigma_{n,2d}$ if and only if $\Vector{\beta} \in \Delta(p)^*$.
This fact was proven in \cite[Corollary 4.9]{Reznick:AGI} for the special case of AGI-forms.

\medskip

We distinguish the simplicial sets that attain one of the two bounds given in \cref{Theorem:Reznick:MMSExists}.
Following Reznick's notation, we call a simplicial set $\Delta$ an \textit{\struc{M-simplex}} if $\Delta^* = \Delta \cup \MidD(\Delta)$, and an \textit{\struc{H-simplex}} if $\Delta^* = \conv(\Delta) \cap \Z^n$.
We motivate this choice of notation in \cref{Example:MMSExample}.
\begin{example}
\label{Example:MMSExample}
Consider the two subsets of $(2\Z)^2$, $\Delta_1 = \left\{ (0,0), (2,4), (4,2) \right\}$ and $\Delta_2 = \left\{ (0,0) , (4,0), (0,4) \right\}$. 
Following \cref{algorithm:MMS1}, we compute that:
\begin{align*}
	\Delta_1^* \ = \ \left\{  (0,0),(1,2),(2,1),(2,4),(3,3),(4,2)  \right\} = \Delta_1 \cup \MidD(\Delta_1), 
\end{align*}
and 
\begin{align*}
	 \Delta_2^* \ = \ \conv(\Delta_2)\cap \Z^2.
\end{align*}
$\Delta_1$ is an example of an $M$-simplex, and arises from the simplicial set associated to the \textit{Motzkin polynomial} given in \cref{Equ:MotzkinPolynomial}. In addition to its historic importance, it is the unique $M$-simplex among the 2-simplicial sets with maximal degree 6.
$\Delta_2$ is an example of $H$-simplex. It arises from a factor-2 scaling of the simplicial set associated to the \textit{Hurwitz form} given in \cref{Equ:HurwitzPolynomial} where $2n =2$.
See \cref{Figure:MMSExample} for the visualization of  $\Delta_1^*$ and $\Delta_2^*$.
\end{example}

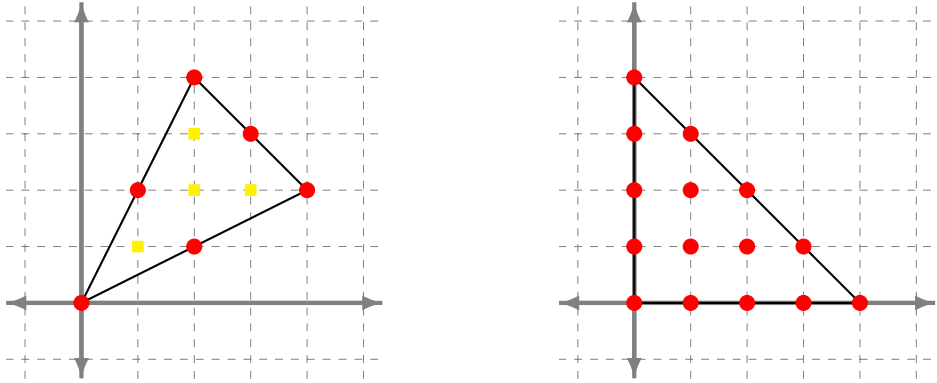
\begin{figure}[t]
	\begin{minipage}{0.45\textwidth}
		\begin{center}
			\begin{tikzpicture}
			\coordinate (Origin)   at (0,0);
			\coordinate (XAxisMin) at (-1,0);
			\coordinate (XAxisMax) at (4,0);
			\coordinate (YAxisMin) at (0,-1);
			\coordinate (YAxisMax) at (0,4);
			\clip (-1,-1) rectangle (4cm,4cm);
			\pgftransformcm{0.75}{0}{0}{0.75}{\pgfpoint{0cm}{0cm}}  
			\draw [ultra thick,gray,-latex] (XAxisMin) -- (XAxisMax);% Draw x axis
			\draw [ultra thick,gray,-latex] (XAxisMax) -- (XAxisMin);% Draw x axis
			\draw [ultra thick,gray,-latex] (YAxisMin) -- (YAxisMax);% Draw y axis
			\draw [ultra thick,gray,-latex] (YAxisMax) -- (YAxisMin);% Draw y axis  
			\draw[style=help lines,dashed] (-14,-14) grid[step=1cm] (14,14);
			% Draws a grid in the new coordinates.
			\draw [thick,black] (0,0) -- (2,4)
			-- (4,2) node {} [below right]--cycle {};
			\node[draw,circle,red,inner sep=2pt,fill] at (2,4) {};
			\node[draw,circle,red,inner sep=2pt,fill] at (4,2) {};
			\node[draw,circle,red,inner sep=2pt,fill] at (0,0) {};
			\node[draw,circle,red,inner sep=2pt,fill] at (1,2) {};
			\node[draw,circle,red,inner sep=2pt,fill] at (2,1) {};
			\node[draw,circle,red,inner sep=2pt,fill] at (3,3) {};
			\node[draw,rectangle,yellow,inner sep=2pt,fill] at (1,1) {};
			\node[draw,rectangle,yellow,inner sep=2pt,fill] at (2,2) {};
			\node[draw,rectangle,yellow,inner sep=2pt,fill] at (2,3) {};
			\node[draw,rectangle,yellow,inner sep=2pt,fill] at (3,2) {};
			\end{tikzpicture}
		\end{center}
	\end{minipage}
	\begin{minipage}{0.45\textwidth}
		\begin{center}
			\begin{tikzpicture}
			\coordinate (Origin)   at (0,0);
			\coordinate (XAxisMin) at (-1,0);
			\coordinate (XAxisMax) at (4,0);
			\coordinate (YAxisMin) at (0,-1);
			\coordinate (YAxisMax) at (0,4);
			\clip (-1,-1) rectangle (4cm,4cm);
			\pgftransformcm{0.75}{0}{0}{0.75}{\pgfpoint{0cm}{0cm}}  
			\draw [ultra thick,gray,-latex] (XAxisMin) -- (XAxisMax);% Draw x axis
			\draw [ultra thick,gray,-latex] (XAxisMax) -- (XAxisMin);% Draw x axis
			\draw [ultra thick,gray,-latex] (YAxisMin) -- (YAxisMax);% Draw y axis
			\draw [ultra thick,gray,-latex] (YAxisMax) -- (YAxisMin);% Draw y axis  
			\draw[style=help lines,dashed] (-14,-14) grid[step=1cm] (14,14);
			% Draws a grid in the new coordinates.
			\draw [thick,black] (0,0) -- (0,4)
			-- (4,0) node {} [below right]--cycle {};
			\foreach \x in {0,1,2,3,4}
			{\foreach \y in {0,1,2,3,4}
				{\ifthenelse{\mycondition{\x}{\y}}{\node[draw,circle,red,inner sep=2pt,fill] at (\x,\y) {} ;}{}
				}}
				\end{tikzpicture}
			\end{center}
		\end{minipage}
	\caption{\small $\Delta_1$(left) and $\Delta_2$(right) in \cref{Example:MMSExample}, red dots indicate the points that are in the MMS and yellow squares indicate the points that are not in the MMS.}
	\label{Figure:MMSExample}
\end{figure}

\begin{example}
	\label{Example:MMSExample1.5}
	Let $\Delta = \{ (0,0,0), (0,2,2), (2,0,2), (2,2,0)\}$, then $\Delta^*$ attains the lower bound in \cref{Theorem:Reznick:MMSExists}
	\begin{align*}
		\Delta^* \  = \ \conv(\Delta) \cap (\Z^n) - \left\{ (1,1,1) \right\} \ = \ \Delta \cup \MidD(\Delta).
	\end{align*}
	One can verify that $\Delta$ is an $M$-simplex also using a result by Bommel \cite[Theorem 3.6]{Bommel:MasterThesis} and the fact that $\frac{1}{2} \Delta$ is a \struc{\textit{distinct pair-sum (dps) polytope}}; see \cite{Choi:Lam:Reznick:DPSPolytopes} for dps polytopes.
	$\Delta$ is the Newton polytope of 
	\begin{align*}
		1 + x_1^2x_2^2 + x_1^2x_3^2 + x_2^2x_3^2 - 4x_1x_2x_3,
	\end{align*}
	another well-known small example of a nonnegative polynomial that is not as sum of squares, see \cite[Equation (1.4)]{Reznick:AGI} and \cite[Equation (2.1)]{Choi:Lam:AnOldQuestionofHilbert}.
\end{example}

For a general $n$-simplicial set $\Delta$, the MMS does not necessarily attain one of the bounds given in \cref{Theorem:Reznick:MMSExists}, see \cref{Example:MMSExample3} and \cref{Example:MMSExample2}.

\begin{example}
	\label{Example:MMSExample3}
	Let $\Delta = \{ (0,0,0,0), (0,0,0,4) , (0,2,2,0), (2,0,2,0), (2,2,0,0)\}$. The convex hull of $\Delta$ contains 22 integral lattice points. 
	Only two of these integral lattice points are not in $\Delta^*$, namely $(1,1,1,0)$ and $(1,1,1,1)$. One can verify this via our software package discussed in \Cref{Section:Implementation}.

\end{example}

\begin{remark}
We point out that, as a result of the computations we perform, \cref{Example:MMSExample1.5} is the unique 3-simplicial set with maximal degree 4 that attains the lower bound in \cref{Theorem:Reznick:MMSExists}. Furthermore,  \cref{Example:MMSExample3} is the only 4-simplicial set with maximal degree at most 4 such that $\Delta^*$ is strictly between the bounds up to coordinate permutations.
\end{remark}

The next example originates from Bruns and Gubeladze \cite{Bruns:Gubeladze:RectengularSimplicial:1999}. 
It points out a connection between normality and the MMS of a lattice simplex. 
\begin{example}
	\label{Example:MMSExample2}
	Let $\Delta = \{ (0,0,0), (4,0,0), (0,6,0), (0,0,10) \} \subset (2\Z)^3$. 
	In \cite[Example 2.2]{Bruns:Gubeladze:RectengularSimplicial:1999}, the authors point out that the lattice polytope $\frac{1}{2}\conv(\Delta)$ is not normal. In particular, $\frac{1}{2}\conv(\Delta)$ is not 2-normal due to the point $\Vector{p} = (1,2,4)$, i.e., there exists no $\Vector{p_1},\Vector{p_2} \in \N^3 \cap \conv(\Delta)$ such that $p_1+p_2 = p$. Consequently, Theorem 5.9 of \cite{Iliman:deWolff:Circuits} implies that $\Delta$ cannot be an $H$-simplex.
	Indeed, $\Vector{p} = (1,2,4)$ is the only point that is not in the MMS:
	\begin{align*}
		\Delta^* \ = \ \conv(\Delta) - \left\{(1,2,4) \right\}.
	\end{align*}
	We visualize $\Delta^*$ in \cref{Figure:MMS3dEx}.
\end{example}
	
\begin{figure}[t]
	\centering
	\ifpictures
	\includegraphics[width=\linewidth]{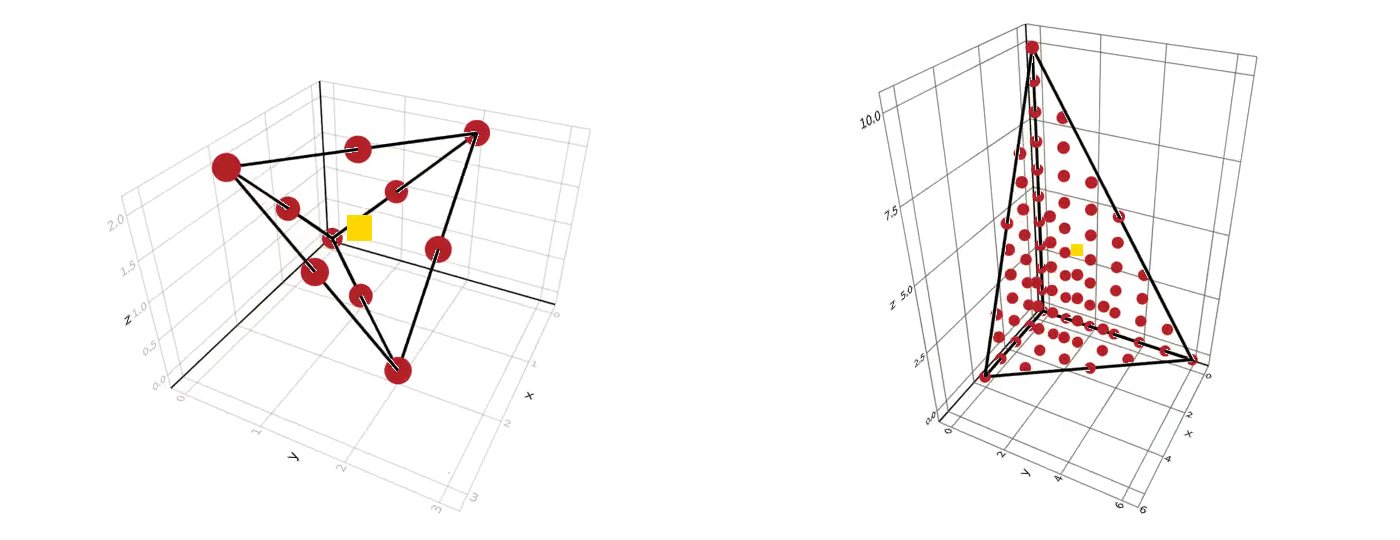}
	\fi
	\caption{ \small MMS of the simplicial sets given in \cref{Example:MMSExample1.5}(left) and \cref{Example:MMSExample2}(right), red dots indicate the points that are in the MMS and yellow squares indicate the points that are not in the MMS.}
	\label{Figure:MMS3dEx}
\end{figure}

For $n = 2$, Reznick stated that $\Delta^*$ is always an $M$-simplex or an $H$-simplex; \cite[Page 9]{Reznick:AGI}.

\begin{conjecture}[Page 9, \cite{Reznick:AGI}]
	\label{Conjecture:ReznickDim2}
	Let $\Delta \subset (2\Z)^2$ be a simplicial set, then $\Delta$ is either an $M$-simplex or an $H$-simplex.
\end{conjecture}

In this 1989 article, Reznick announced a proof for this claim, and another important result \cite[Proposition 2.7]{Reznick:AGI}, but the particular article was not finished.
% % % Unfortunately, 30 years ago, he decided to work on other problems instead. 
Very recently, Powers and Reznick proved \cite[Proposition 2.7]{Reznick:AGI} in \cite{Reznick:Powers:NoteOnMMS}. However, after consulting with the authors we reached a consensus that, their results in \cite{Reznick:Powers:NoteOnMMS} do not solve \cref{Conjecture:ReznickDim2}.

In \cite{Iliman:deWolff:Circuits}, the result Iliman and the third author implies that the most of the 2-simplicial sets are indeed $H$-simplices.

\begin{corollary}[Iliman, dW. \cite{Iliman:deWolff:Circuits}]
	\label{Corollary:Iliman:deWolff:Hsimplex}
	For any 2-simplicial set $\Delta \subset \Z^2$, if $\frac{1}{2} \conv(\Delta)$ has at least 4 integral boundary points, then $\Delta$ is an $H$-simplex 
\end{corollary}

%Reznick further states in \cite[Page 462]{Reznick:AGI} that, with referring to same forthcoming paper, every 2-simplicial set is either an $H$-simplex or originates from a 1-parameter family of trellises.
%\begin{conjecture}
%	\label{Conjecture:UpConjecture}
%	Let $\Delta \subset (2\Z)^2$ be a simplicial set, and
%	\begin{align*}
%		\struc{\Up} : = \{(0,0,2p),(2,2p-2,0),(4,2,2p-6)\}.
%	\end{align*}
%	Then $\Delta$ is either an $M$-simplex or ... (I couldn't find a good formulation for here)
%	\oguzhan{Maybe we should avoid stating this as a conjecture, since there is not a good formulation of the statement(at least until we have an answe from Bruce).}
%\end{conjecture}

In this article, we provide an experimental proof for \cref{Conjecture:ReznickDim2} for simplices of degree at most 150 see \Cref{Subsection:EvaluationOfTheExperiment}.

\section{Theoretical Results}
\label{Section:TheoreticalResults}

In this section we present our theoretical results on maximal mediated sets.	
We start with motivating the central definition of this section.
Let $f \in P_{n,2d}$ be a circuit polynomial supported on a circuit with vertex set $\Delta = \vertices{\newton{f}}$ and inner term $\Vector{\beta}$.
The maximal mediated set associated to $f$, $\Delta^*$, is the set of choices for $\Vector{\beta}$ that ensure $f \in \Sigma_{n,2d}$.
Therefore, the density of $\Delta^*$ in $\conv(\Delta) \cap \Z^n$ is a measure of how likely $f$ is to be a SOS polynomial.
Due to \cref{Theorem:Reznick:MMSExists}, we have
\begin{align*}
\Delta \cup \MidD(\Delta) \subseteq \Delta^*.
\end{align*}
Thus, we exclude these points that are a priori in the MMS while we describe the density of $\Delta^*$ in $\conv(\Delta) \cap \Z^n$.
\begin{definition}
	\label{Definition:hratio}
	Given a simplex $\Delta \subseteq (2\Z)^n$, we define the \textit{\struc{$h$-ratio}} of $\Delta$ as follows:	
	\begin{align*}
	\struc{\Hratio(\Delta)} \ = \begin{cases}
	\ \dfrac{\#\left( \Delta^* -(\Delta \cup \MidD(\Delta)) \right)}{\# \left( (\conv(\Delta)\cap \Z^n)-(\Delta \cup \MidD(\Delta)) \right)} &\text{if } (\conv(\Delta)\cap \Z^n) \neq (\Delta \cup \MidD(\Delta)),\\
	\\
	\ 1  &\text{otherwise}
	\end{cases}	 
	\end{align*} 
\end{definition}

Note that in terms of polynomials, the $h$-ratio of the Newton polytope is an indicator for the likelihood of a nonnegative circuit polynomial to be SOS.

\subsection{MMS Preserving Group Actions}
\label{Subsection:MMSPreservingGroupActions}

We aim to classify $n$-simplicial sets with maximal degree $2d$ according to their $h$-ratio.
Besides understanding what properties determine the $h$-ratio, another benefit of such a classification is the opportunity to reduce the size of the database of maximal mediated sets by storing one representative of each class only. 
Therefore, in this section we study the maps from $\R^n$ to $\R^n$ that preserve the maximal mediated set structure. In \cite[Page 445]{Reznick:AGI} the author points out that, the maps that respect the MMS structure are necessarily linear maps in the context of trellises.
We provide a rigorous proof for this observation in the setting of simplicial sets and $h$-ratios.
In particular, we are interested in the following maps.
% % % $T: \R^n \to \R^n$ such that for any $k$-simplicial set $\Delta$, $T(\Delta)\subseteq(2\Z)^n$ is $k$-simplicial and $\Hratio(\Delta) = \Hratio(T(\Delta))$.

\begin{definition}
\label{Definition:MMSpres}
A function $T : \R^n \to \R^n$ is called \textit{\struc{maximal mediated set preserving (MMS preserving)}} if and only if it satisfies the following properties for every $k$-dimensional simplicial set $\Delta \subseteq (2\Z)^n$:
\begin{enumerate}
\item $T(\Delta) \subseteq (2\Z)^n$ is a $k$-dimensional simplicial set in $\Rn$.
\item For every $\Vector{q} \in T(\Delta)^*$, there exists a unique $\Vector{p} \in \Delta^*$ such that $T(\Vector{p}) = \Vector{q}$.
\item For every $\Vector{q}\in (\conv(T(\Delta))\cap \Z^n)$, there exists a unique $\Vector{p}\in (\conv(\Delta)\cap \Z^n)$ such that $T(\Vector{p})=\Vector{q}$.
\end{enumerate}
\end{definition}
\cref{Definition:MMSpres} has some immediate implications for every MMS preserving function $T$. 
Due to the first property with $k = 0$ we have: 
\begin{align}
\label{Eqn:MMSPres2Z}
\Vector{p}\in (2\Z)^n \implies T(\Vector{p})\in (2\Z)^n.
\end{align}
The second and the third property respectively ensure for every $k$-dimensional simplicial set $\Delta$ that 
\begin{align*}
	\#\Delta^* \ = \ \#T(\Delta)^* \ \text{ and } \ \#(\conv(\Delta) \cap \Z^n) \ = \ \#(\conv(T(\Delta))\cap \Z^n).
\end{align*}
Hence, the $h$-ratio is invariant under a maximal mediated set preserving function $T$. 
 
Note that the property (1) of \cref{Definition:MMSpres} is equivalent to $T$ mapping any $k$-dimensional affine independent subset of $(2\Z)^n$ to a $k$-dimensional affine independent set of $(2\Z)^n$.
This means the restriction of $T$ to $(2\Z)^n$ is an affine transformation of $(2\Z)^n$.
The next proposition generalizes this result to $\R^n$.

\begin{proposition}
\label{Proposition:MMSPresAffine}
Let $T:\R^n \to \R^n$ be a maximal mediated set preserving function, then $T$ is a unimodular affine transformation. More specifically, we have
\begin{align}
T(\Vector{x}) \ = \ A_T\Vector{x}+\Vector{b}_T
\label{Eqn:AffineTrans}
\end{align}
with $\Vector{b}_T \in (2\Z)^n$, $A_T \in \Z^{n \times n}$, and $\det(A) = \pm 1$.
\end{proposition}

\begin{proof}
Assume that $T$ is MMS preserving and let $K \subset \R^n$ be a collection of affine independent vectors. First, we show that $T(K)$ is also affine independent. This is true by \cref{Definition:MMSpres} for $K \subset (2\Z)^n$ and thus also for $K \subset \Q^n$ with a suitable scaling of the elements of $K$.
Since $\Q^n$ is a dense subset of $\R^n$, we conclude that $T$ is an affine transformation over $\R^n$.

Due to \cref{Eqn:MMSPres2Z} we know $T\left((2\Z)^n \right) = (2\Z)^n$. 
In particular, we have,
\begin{align*}
	T(\Vector{0}) \ = \ \Vector{b}_T \in (2\Z)^n.
\end{align*}

Now we prove that $T(\Z^n)=\Z^n$, i.e., $A_T \in \Z^{n \times n}$.
Let $\Vector{p} = (p_1, \dots ,p_n) \in \Z^n$, and $J \subset [n]$ be the subset of indices where $p_i$ is odd.
We define two points $\Vector{p^+}$ and $\Vector{p^-}$ as follows:
\begin{align*}
p^+_i \ = \
\begin{cases}
p_i, \ \text{if } i\notin J\\
p_i+1, \ \text{if } i\in J
\end{cases}
\text{, and \ }
p^-_i \ = \
\begin{cases}
p_i, \ \text{if } i\notin J\\
p_i-1, \ \text{if } i\in J
\end{cases} 
\end{align*}
for all $i \in [n]$. Observe that $\Vector{p^+},\Vector{p^-} \in (2\Z)^n$ and $\Vector{p}  = \dfrac{\Vector{p^+}+\Vector{p^-}}{2}$.
Since $T$ is affine we have,
\begin{align*}
T(\Vector{p}) \ = \ T \left(\frac{\Vector{p^+}+\Vector{p^-}}{2}\right) \ = \ \frac{T(\Vector{p^+})+T(\Vector{p^-})}{2}.
\end{align*}
Due to \cref{Eqn:MMSPres2Z}, we have $T(\Vector{p^+}), T(\Vector{p^-}) \in (2\Z)^n$, and thus $T(\Vector{p}) \in \Z^n$.
Therefore, $T(\Z^n) = \Z^n$ and hence $A_T \in \Z^{n \times n}$.

Finally, we show $\det(A_T) = \pm1$. By part (1) of \cref{Definition:MMSpres}, $T$ maps any set of $n$ linearly independent vectors to another set of $n$ linearly independent vectors. Thus, $\det(A_T) \neq 0$. By part (3) of \cref{Definition:MMSpres} volumes of simplices are preserved under $T$, and hence $\det(A_T) = \pm 1$.
\end{proof}
Recall that any affine linear transformation can be represented as an element of the group $\R^n \rtimes \GL(\R^n)$. 
In particular, if $T$ is MMS preserving, then we have $T \in (2\Z)^n \rtimes \GL(\Z^n)$ due to \cref{Proposition:MMSPresAffine}. 
This motivates the following definition.

\begin{definition}
We define the \struc{\textit{maximum mediated set preserving group}} of $\R^n$, $\MMSPreservingGroup$ as follows:
\begin{align*}
\struc{\MMSPreservingGroup} \ = \ \left(\left\{ T\in (2\Z)^n\rtimes \GL(\Z^n) \ \mid \text{ $T$ is maximal mediated set preserving }\right\},\circ\right),
\end{align*}
where $\struc{\circ}$ denotes the usual composition of functions.
\end{definition}

Obviously, $\MMSPreservingGroup$ is a subgroup of $(2\Z)^n \rtimes \GL(\Z^n)$. 
In the next theorem we show that it is in fact the full group.

\begin{theorem}
\label{Theorem:MMSPresGroup}
$T \in \MMSPreservingGroup$ if and only if $T \in (2\Z)^n \rtimes \GL(\Z^n)$, i.e., $T : \R^n \to \R^n$ is MMS preserving if and only if
\begin{align*}
T(\Vector{x}) \ = \ A_T\Vector{x}+\Vector{b}_T
\end{align*}
is a unimodular affine transformation with $\Vector{b}_T \in (2\Z)^n$.
\end{theorem}

\begin{proof}
Let $T \in \MMSPreservingGroup$ , the only if part of the theorem follows from \cref{Proposition:MMSPresAffine}.
For the converse, assume that $T : \R^n \to \Rn$ is a unimodular affine transformation with $\Vector{b}_T \in (2\Z)^n$.
We need to show that $T$ is MMS preserving.
Let $\Delta = \{ \Vector{v_0}, \dots, \Vector{v_k} \}$ be a $k$-simplicial set.
By definition, $\Vector{v_0}, \dots, \Vector{v_k}$ are affine independent, and, since $T$ is affine, $T(\Delta)$ is a set of $k+1$ affine independent vectors, i.e., a $k$-simplicial set.
Furthermore, if $\Delta \subset (2\Z)^n$, then $T(\Delta) \subset (2\Z)^n$ because
\begin{align*}
T(\Vector{v_i}) \ = \ A_T \Vector{v_i} + \Vector{b}_T \in (2\Z)^n.
\end{align*}
This implies part (1) of \cref{Definition:MMSpres}.

Next we show part (3) of \cref{Definition:MMSpres}:
Let $\Vector{q} \in \conv(T(\Delta))\cap \Z^n$.
Since $T$ is unimodular, $T^{-1}$ exists and is also unimodular.
This implies that $T^{-1}(\Vector{q}) \in \Z^n$.
Furthermore since $\Vector{q} \in \conv(T(\Vector{v_0}),\dots, T(\Vector{v_n}))$, there exists $\lambda_0,\ldots,\lambda_n \in \R_{> 0}$ with $\sum_{i=0}^n \lambda_i = 1$ and $\Vector{q} = \sum_{i=0}^n \lambda_i T(\Vector{v_i})$.
Then,
\begin{align*}
T^{-1}(\Vector{q}) & \ = \ T^{-1} \left( \sum_{i=0}^n \lambda_iT(\Vector{v_i}) \right)
						  \ = \ \sum_{i=0}^n \lambda_iT^{-1}(T(\Vector{v_i}))
						  \ = \ \sum_{i=0}^n \lambda_i\Vector{v_i} \in \conv(\Vector{v_0},\dots,\Vector{v_k}).
\end{align*}
Therefore, for all $\Vector{q} \in \conv(T(\Delta)) \cap \Z^n$, there exists a unique $\Vector{p} = T^{-1} (\Vector{q})\in \conv(\Delta) \cap \Z^n$. 

Finally, we show part (2) of \cref{Definition:MMSpres}:
Since $T$ maps $\conv(\Delta) \cap \Z^n$ to $\conv(T(\Delta)) \cap \Z^n$ bijectively, we are done if we show that $\Vector{p} \in \Delta^*$ if and only if $T(\Vector{p}) \in T(\Delta)^*$.
Define the sets $\struc{U^0} = \conv(\Delta) \cap \Z^n$, $\struc{V^0} = \conv(T(\Delta)) \cap \Z^n$, and recursively define the sets $U^k$ and $V^k$ as follows:
\begin{align*}
\struc{U^k} \ = \ \MidD(U^{k-1}) \cup \Delta, \qquad \struc{V^k} \ = \ \MidD(V^{k-1}) \cup T(\Delta).
\end{align*}
By \cref{algorithm:MMS1}, we know that $\Vector{p} \in \Delta^*$ if and only if $\Vector{p} \in U^k$ for all $k$ and $T(\Vector{p}) \in T(\Delta)^*$ if and only if $T(\Vector{p}) \in V^k$ for all $k$.
We claim that $T(U^k) = V^k$ is a bijection for all $k\in \N$ and we argue by induction over $k$. 
We already know that $T(U^0)=V^0$ is a bijection.
Now assume that $T$ sends $U^k$ to $V^k$ bijectively, and let $\Vector{q}$ be a point in $V^{k+1}$. 
Then either $\Vector{q} \in \MidD(V^k)$ or $\Vector{q} \in T(\Delta)$.
On the one hand, if $\Vector{q} \in T(\Delta)$, then $\Vector{q}$ is a vertex of $\conv(T(\Delta))$ and hence there exists a unique $\Vector{p} \in \Delta$ with $T(\Vector{p}) \ = \ \Vector{q}$ by part (1) of \cref{Definition:MMSpres}, which we have already shown to hold.

On the other hand, if $\Vector{q} \in \MidD(V^k)$, then there exist distinct $\Vector{\tilde{s}}, \Vector{\tilde{t}} \in V^k$ such that
\begin{align*}
\Vector{q} \ = \ \frac{1}{2}(\Vector{\tilde{s}} + \Vector{\tilde{t}}).
\end{align*} 
By the induction assumption, $\Vector{\tilde{s}}$ and $\Vector{\tilde{t}}$ have unique preimages $\Vector{s}$ and $\Vector{t}$ in $U^k$ respectively.
Since $T^{-1}$ is affine linear, we obtain a unique
\begin{align*}
\Vector{p} \ = \ T^{-1}(\Vector{q}) \ = \ T^{-1}\left(\frac{1}{2}(\Vector{\tilde{s}} + \Vector{\tilde{t}})\right) \ = \ \frac{1}{2}\left(T^{-1}(\Vector{\tilde{s}}) + T^{-1}(\Vector{\tilde{t}})\right) \ = \ \frac{1}{2}(\Vector{s} + \Vector{t}).
\end{align*}
Thus, for all $k$, $T$ maps $U^k$ to $V^k$ bijectively.
Hence, we conclude $\Vector{p} \in \Delta^*$ if and only if $T(\Vector{p}) \in T(\Delta)^*$.
\end{proof}

We present a key corollary of \cref{Theorem:MMSPresGroup}:
If we exclude the translations, then an MMS preserving function is given by a unimodular matrix $A_T$.
We show that the row span of $A_T M_{\Delta}$ yields, up to a permutation of the coordinates, the same lattice $L_{\Delta}$ as the row span of $M_{\Delta}$.
Thus, the maximal mediated set of a $k$-simplicial set $\Delta$ containing the origin is an invariant of the lattice generated by the rows of $M_{\Delta}$.

\begin{corollary}
\label{Corollary:MMSisInvariantofLattice}
Given two $k$-simplicial sets $\Delta_1 = \left\{\Vector{0}, \Vector{v_1}, \dots, \Vector{v_k} \right\}$ and $\Delta_2 = \left\{ \Vector{0}, \Vector{u_1}, \dots, \Vector{u_k} \right\}$, there exists a $T\in \MMSPreservingGroup$ with $\Vector{b}_T = \Vector{0}$ such that
\begin{align*}
T(\Delta_1)^* \ = \ \Delta_2^*
\end{align*}
if and only if the lattices $L_{\Delta_1}$ and $L_{\Delta_2}$ share the same Hermite Normal Form up to a permutation of columns.  
\end{corollary}

\begin{proof}
If there exists $T \in \MMSPreservingGroup$ with $\Vector{b}_T= \Vector{0}$ and $T(\Delta_1)^* = \Delta_2^*$, then $M_{\Delta_1} = A_T M_{\Delta_2}$ where $A_T \in \GL(\Z^n)$.
Therefore, the $\Z$-row span of $M_{\Delta_1}$ and $M_{\Delta_2}$ yields, up to a permutation of the coordinates, the same lattice.
In converse, if the lattices $L_{\Delta_1}$ and $L_{\Delta_2}$ share the same Hermite Normal Form up to a permutation of columns, then the $\Z$-row spans of $M_{\Delta_1}$ and $M_{\Delta_2}$ coincide up to a permutation of columns. Hence, there exists $A \in \GL(\Z^n)$ such that $M_{\Delta_1} = A M_{\Delta_2}$.
Thus, the transformation $T(\Vector{x}) = A \Vector{x}$ is MMS preserving.
\end{proof}

\begin{example}
\label{Example:MMSInvariantofLattice}
Let 
\begin{align*}
\Vector{0} =\begin{bmatrix}
0 \\
0
\end{bmatrix},
\Vector{v_1} =\begin{bmatrix}
2 \\
4
\end{bmatrix},
\Vector{v_2} =\begin{bmatrix}
4 \\
2
\end{bmatrix},
\Vector{u_1} =\begin{bmatrix}
2 \\
0
\end{bmatrix},
\Vector{u_2} =\begin{bmatrix}
4 \\
6
\end{bmatrix},
\end{align*}
and $\Delta_1, \Delta_2 \subset (2\Z)^2$ be given as $\Delta_1 \ = \ \left\{\Vector{0},\Vector{v_1} , \Vector{v_2} \right\}$ and $\Delta_2 \ = \ \left\{\Vector{0},\Vector{u_1} , \Vector{u_2} \right\}$.
Then we write the matrices $M_{\Delta_1}$ and $M_{\Delta_2}$ with the given order(see \cref{Figure:CorollaryEx1});
\begin{align*}
M_{\Delta_1} \ = \ \begin{bmatrix}
2& 4 \\
4& 2
\end{bmatrix} \quad
\text{ and } \quad
M_{\Delta_2} \ = \ \begin{bmatrix}
2& 4 \\
0& 6
\end{bmatrix}.
\end{align*}

Note that $M_{\Delta_2} \ = \ A M_{\Delta_1}$ , where 
\begin{align*}
A \ = \ \begin{bmatrix}
1&  0 \\
2& -1
\end{bmatrix}
\end{align*} 
is a unimodular matrix.
Indeed, the lattices corresponding to the row spans of $M_{\Delta_1}$ and $M_{\Delta_1}$ share the same Hermite Normal Form
\begin{align*}
H_1 = \begin{bmatrix}
2& 4 \\
0& 6
\end{bmatrix}
\end{align*}
and they are the same lattice, i.e., 
\begin{align*}
L_{\Delta_1} \ = \ \left\langle \begin{bmatrix}
2 \\
4
\end{bmatrix}, \begin{bmatrix}
4 \\
2
\end{bmatrix}\right\rangle \ = \
\left\langle \begin{bmatrix}
2 \\
4
\end{bmatrix}, \begin{bmatrix}
0 \\
6
\end{bmatrix}\right\rangle
\ = \ L_{\Delta_2}.
\end{align*}
\end{example}

%=============================================
\begin{figure}[t]
	\ifpictures
	\centering
	\includegraphics[scale=0.45]{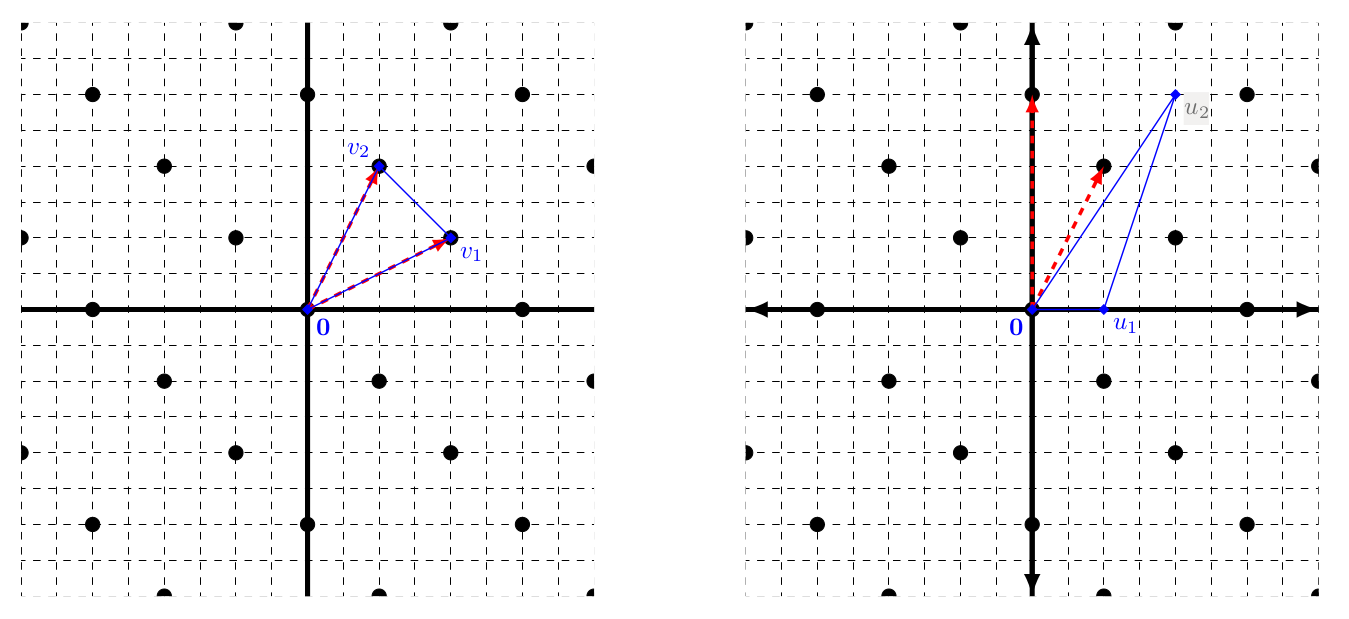}
	\fi
	\caption{\small Black points show the lattice generated by the rows of $M_{\Delta_1}$ (left) and $M_{\Delta_2}$ (right) where $\Delta_1$ and $\Delta_2$ are given as in \cref{Example:MMSInvariantofLattice}. The (blue) triangles visualize the 2-simplicial sets $\Delta_1$, $\Delta_2$ and the (red) dashed vectors visualizes the row span of $M_{\Delta_1}$ and $M_{\Delta_2}$.}
	\label{Figure:CorollaryEx1}
\end{figure}
%=============================================

Note that in the \cref{Example:MMSInvariantofLattice} we obtain identical lattices. We show that this is not always the case in the \cref{Example:MMSInvariantofLattice2}.

\begin{example}
\label{Example:MMSInvariantofLattice2}
Let 
\begin{align*}
\Vector{0} =\begin{bmatrix}
0 \\
0
\end{bmatrix},
\Vector{v_1} = \Vector{u_2} =\begin{bmatrix}
2 \\
2
\end{bmatrix},
\Vector{v_2} = \Vector{u_1} =\begin{bmatrix}
0 \\
6
\end{bmatrix}
\end{align*}
and $\Delta_1, \Delta_2 \subset (2\Z)^2$ be given as $\Delta_1  =  \left\{\Vector{0},\Vector{v_1} , \Vector{v_2} \right\}$  and $\Delta_2  =  \left\{\Vector{0},\Vector{u_1} , \Vector{u_2} \right\}$.
$\Delta_1$ and $\Delta_2$ have same maximal mediated sets since they correspond to the same simplex with different vertex order(see \cref{Figure:CorllaryEx2}).
Again we write the matrices, 
\begin{align*}
M_{\Delta_1} \ = \ \begin{bmatrix}
2 & 0  \\
2 & 6  
\end{bmatrix} \quad 
\text{ and }\quad 
M_{\Delta_2} \ = \ \begin{bmatrix}
0 & 2 \\
6 & 2 
\end{bmatrix}.
\end{align*}
If we denote the Hermite Normal Form of $M_{\Delta_1}$ and $M_{\Delta_2}$ as $H_1$ and $H_2$ respectively, then we have 
\begin{align*}
H_1 =  \begin{bmatrix}
2 & 0  \\
0 & 6  
\end{bmatrix} =\begin{bmatrix}
1 & 0  \\
-1 & 1 
\end{bmatrix}  
M_{\Delta_1}  
\end{align*}
and
\begin{align*}
H_2 =  \begin{bmatrix}
6 & 0  \\
0 & 2  
\end{bmatrix} =\begin{bmatrix}
-1 & 1  \\
1 & 0 
\end{bmatrix}  
M_{\Delta_2}.
\end{align*}
In this case Hermite Normal Forms are equal up to permutation of columns. 
Therefore, the lattice generated by the row spans of $M_{\Delta_1}$ and $M_{\Delta_2}$ are not identical, but they are isomorphic. This isomorphism is given by a permutation of the coordinates of the lattice.
\end{example}

\begin{figure}[t]
\ifpictures
\includegraphics[scale=0.45]{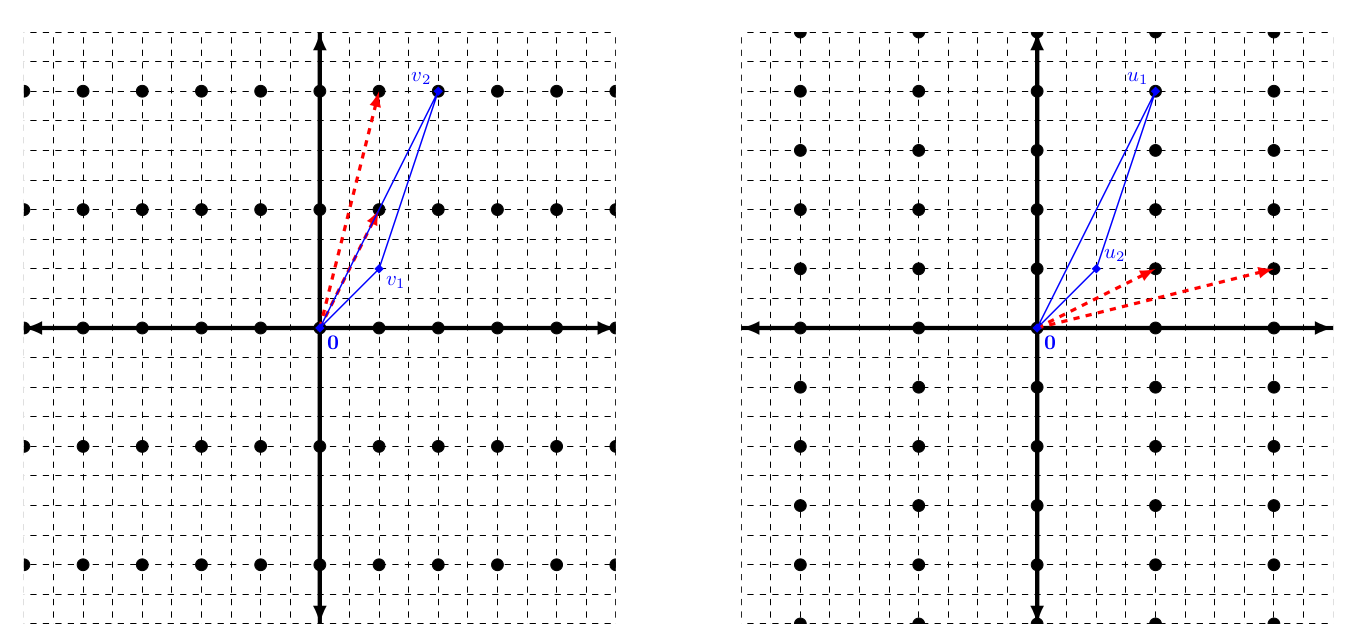}
\fi
\caption{\small Black points show the lattice generated by the rows of $M_{\Delta_1}$ (left) and $M_{\Delta_2}$ (right) where $\Delta_1$ and $\Delta_2$ are given as in \cref{Example:MMSInvariantofLattice2}. The (blue) triangles visualize the 2-simplicial sets $\Delta_1$, $\Delta_2$ and the (red) dashed vectors visualizes the row span of $M_{\Delta_1}$ and $M_{\Delta_2}$.}
\label{Figure:CorllaryEx2}
\end{figure}

\subsection{SONC with Simplex Newton Polytope}~
\label{Subsection:SONCWithSimplexNewtonPolytope}
We generalize \cite[Theorem 5.2]{Iliman:deWolff:Circuits} to sums of nonnegative circuit polynomials with simplex Newton polytope in this subsection.
The next proof heavily relies on the Gram matrix method; see e.g., \cite[Page 20]{Iliman:deWolff:Circuits} for an overview.
Let $\struc{\R[\Vector{x}]_d}$ denote the polynomials in $\R[\Vector{x}]$ with total degree at most $d$. 
For $f \in \R[\Vector{x}]_{d}$ and $\Vector{\alpha} \in \N^n$ , let $\struc{\coeff(f, \Vector{\alpha})}$ denote the coefficient of the term $\Vector{x}^{\Vector{\alpha}}$ in $f$.

% % % GRAM MATRIX METHOD
% % % If $f \in \Sigma_{n,2d}$ then there exists $h_1, \dots ,h_k \in \R^n[\Vector{x}]_d$ such that 
% % % \begin{align}
% % % \label{Eqn:SoSRepresentation}
% % % f \ = \ \sum_{i = 1}^{k} h_i^2.
% % % \end{align}
% % % Given $\Vector{\beta} \in \N_d^n$, let $B(\Vector{\beta}) = \left( c(h_1, \Vector{\beta}), \dots, c(h_k, \Vector{\beta}) \right)$ and define
% % % \begin{align*}
% % % G\left( \Vector{\beta}, \Vector{\beta'} \right) \ := \ B(\Vector{\beta}) \cdot B(\Vector{\beta'}) \ = \ \sum_{i = 1}^{k} c(h_i, \Vector{\beta}) c(h_i, \Vector{\beta'}).
% % % \end{align*}
% % % The matrix $[G(\Vector{\beta},\Vector{\beta'})]_{\Vector{\beta},\Vector{\beta'} \in N_d^n}$ is the Gram matrix of $f$ with respect to $h_1, \dots, h_k$.
% % % Using the Gram matrix, we compare the coefficients in \cref{Eqn:SoSRepresentation} for $\Vector{\alpha} \in \N_{2d}^n$: 
% % % \begin{align}
% % % \label{Eqn:CoefCompare}
% % % c(f, \Vector{\alpha}) = \sum_{\Vector{\beta} + \Vector{\beta'} = \Vector{\alpha}} G\left( \Vector{\beta}, \Vector{\beta'} \right) =  \sum_{\Vector{\beta} \in \N_d^n} G\left( \Vector{\beta},\Vector{\alpha} - \Vector{\beta} \right)
% % % \end{align}

\begin{theorem}
\label{Theorem:SONCwithSimplexNewt}
Let $\Delta = \{\Vector{0}, \Vector{\alpha(1)} , \dots , \Vector{\alpha(n)} \} \subset (2\Z)^n$ be a full dimensional simplicial set, $Y = \{ \Vector{\beta_1}, \dots, \Vector{\beta_m} \} \subseteq \Int \left(\conv(\Delta) \cap \Z^n \right)$ be a set of points.
Let $f = \lambda_0 + \sum_{i=0}^n a_i \Vector{x}^{\Vector{\alpha(i)}} + \sum_{\Vector{\beta} \in Y} b_{\Vector{\beta}} \Vector{x}^{\Vector{\beta}}$ be a SONC with support $\Delta \cup Y$, $a_0, \ldots, a_n > 0$,
such that for all $\Vector{\beta} \in Y$,  $b_{\Vector{\beta}} < 0$ or $\Vector{\beta} \notin (2\Z)^n$.
Then $f$ is a sum of squares if and only if every $\Vector{\beta} \in Y$ satisfies $\Vector{\beta} \in \Delta^*$.
% % % \begin{align*}
% % % f \in \Sigma_{n,2d} \iff \text{for all} \ \Vector{\beta} \in Y , \ \Vector{\beta} \in \Delta^* \text \ \ {or} \ \ \Vector{\beta} \in (2\Z)^n \ {and} \ b_{\Vector{\beta}}>0 
% % % \end{align*}
\end{theorem}

This theorem generalizes \cite[Theorem 5.2]{Iliman:deWolff:Circuits}, which states the same result for the case $\# Y = 1$.

\begin{proof}
First, assume that $f$ admits a SONC decomposition $f = \sum_{\Vector{\beta} \in Y} s_{\Vector{\beta} }$
where $s_{\Vector{\beta}}$ is a nonnegative circuit polynomial with the support $\Delta \cup \{ \Vector{\beta} \}$.
Since $s_{\Vector{\beta}}$ is a nonnegative circuit polynomial satisfying $\Delta(s_{\Vector{\beta}})^* = \Delta(f)^* = \Delta^*$, and since we have $\Vector{\beta} \in \Delta^*$ by assumption, \cite[Theorem 5.2]{Iliman:deWolff:Circuits} implies that $s_{\Vector{\beta}} \in \Sigma_{n,2d}$.
Thus, it follows that $f\in \Sigma_{n,2d}$.
	
For the converse, assume that $f\in \Sigma_{n,2d}$. 
We claim that given a $\Vector{\beta} \in Y$, if 
\begin{align}
\label{Eqn:GoodTerm}
\Vector{\beta} \notin (2\Z)^n \text{ or } b_{\Vector{\beta}} \leq 0
\end{align}
then $\Vector{\beta}  \in \Delta^*$.
Since $f \in \Sigma_{n,2d}$, it has a SOS decomposition $f = \sum_{i=i}^{k} h_i^2$.
Define the set 
\begin{align*}
	\struc{M} \ = \ \left\{ \Vector{\gamma}\in \N^n_d \ : \ \text{there exists an } i \in [k] \text{ with } \coeff(h_i,\Vector{\gamma}) \neq 0  \right\}.
\end{align*}
For every $\Vector{\beta} \in Y$ we define the set 
\begin{align*}
	\struc{L_{\Vector{\beta}}} \ = \ 2M \cup \Delta \cup {\Vector{\beta}}.
\end{align*}
We can assume $b_{\Vector{\beta}} < 0$: if $\Vector{\beta} \in (2\Z)^n$, then $b_{\Vector{\beta}} < 0$ by \cref{Eqn:GoodTerm}. Assume that $b_{\Vector{\beta}} > 0$ and that there exists an odd entry $\beta_j$ of $\Vector{\beta}$ for some $j \in [n]$. After a transformation $\tau_j: x_j \mapsto -x_j$ we can consider $b_{\Vector{\beta}} < 0$, while $\tau_j$ leaves $\sum_{i = 1}^k h_i^2$ invariant; see also \cite[proof of Theorem 5.2]{Iliman:deWolff:Circuits} and e.g., \cite{Blekherman:Parrilo:Thomas:SDOptAndConvAlgGeo}.
With $b_{\Vector{\beta}} < 0$ we obtain that $L_{\Vector{\beta}}$ is $\Delta$-mediated following verbatim the first part of the proof of \cite[Theorem 5.2]{Iliman:deWolff:Circuits}.
This completes the proof since every $\Delta$-mediated set is contained in $\Delta^*$ by \cref{Theorem:Reznick:MMSExists}.

\end{proof}

Given $f = \lambda_0 + \sum_{i=0}^n a_i \Vector{x}^{\Vector{\alpha(i)}} + \sum_{\Vector{\beta} \in Y} b_{\beta} \Vector{x}^{\Vector{\beta}} \in \R[x_1, \dots, x_n]_{2d}$ such that  $a_i>0$, $b_{\Vector{\beta}}<0$ and $\newton{f}$ is a simplex, due to \cite[Theorem 5.5]{Iliman:deWolff:Circuits} we know that
\begin{align*}
f \in P_{n,2d} \iff f \text{ is SONC}.
\end{align*}

The following corollary concerns two relaxations in polynomial optimization.
First, we recall the following two quantities:
\begin{align*}
\struc{f_{\text{SOS}}} & \ := \ \max \{\lambda \ : \ f-\lambda \text{ is SOS} \} \\
\struc{f_{\text{SONC}}} & \ := \ \max \{\lambda \ : \ f-\lambda \text{ is SONC} \}.
\end{align*}
Both $ f_{\text{SOS}} $ and $f_{\text{SONC}}$ are lower bounds for $f^* := \min \{f(\Vector{x}) \ : \ \Vector{x} \in \R^n \}$. 
Next, we present a corollary of \cref{Theorem:SONCwithSimplexNewt} which generalizes the part (2) of \cite[Corollary 3.6]{Iliman:deWolff:LowerBoundsforSimplexNewtPoly}

\begin{corollary}
Let $f$ be as given in \cref{Theorem:SONCwithSimplexNewt}, then 
\begin{align*}
f_{\text{SOS}} \ = \ f^* \iff \text{for all} \ \Vector{\beta} \in Y , \ \Vector{\beta} \in \Delta^* \text{ or } \Vector{\beta} \in (2\Z)^n \text{ and } b_{\Vector{\beta}}>0 
\end{align*}
Furthermore, if there exists $\Vector{v} \in (\R^*)^n$ such that $b_{\Vector{\beta}}<0$ for all $\Vector{\beta} \in Y$ then
\begin{align*}
f_{\text{SOS}} \ = \ f_{\text{SONC}} = f^*
\end{align*} 
\end{corollary}

\begin{proof}
Note that $f_{\text{SOS}} = f^*$ if and only if $f-f^* \in \Sigma_{n,2d}$.
Since $f^* \in \R$, subtracting it from $f$ does not effect the support of $f$.
In addition, $f-f^*$ is SONC because $f$ is SONC.
Therefore \cref{Theorem:SONCwithSimplexNewt} implies that $f-f^* \in \Sigma_{n,2d}$ if and only if $\Vector{\beta} \in Y , \ \Vector{\beta} \in \Delta^*$ or $\Vector{\beta} \in (2\Z)^n \text{ and } b_{\Vector{\beta}}>0$.
Furthermore, if there exists $\Vector{v} \in (\R^*)^n$ such that $b_{\Vector{\beta}}<0$ for all $\Vector{\beta} \in Y$, then $f_{\rm SONC} = f^*$ due to \cite[Corollary 3.6]{Iliman:deWolff:LowerBoundsforSimplexNewtPoly}.
\end{proof}

%=============================================
\section{Implementation}
\label{Section:Implementation}

A major goal of this article is to generate a database by classifying all maximal mediated sets of $n$-simplicial sets with maximal degree $2d$ for $n$ and $2d$ as large as possible. 
Here we describe the underlying implementation for this computation.
The implementation was done in \texttt{C++} using the \textsc{Polymake} \cite{Ewgeniy:Joswig:polymake:2000} software package; its source code can be obtained via

\begin{center}
	\href{https://polymake.org/doku.php/extensions/max\_mediated\_sets}{https://polymake.org/doku.php/extensions/max\_mediated\_sets}
\end{center}

There are three major parts to the process:
\begin{enumerate}
	\item Enumerating simplices,
	\item classifying simplices, and
	\item computing the maximal mediated set.
\end{enumerate}

\subsection{Enumerating Simplices}
Assume that $n$ and $2d$ are fixed. Following \cref{Corollary:MMSisInvariantofLattice}, we restrict to the simplices containing a vertex at the origin.
Let
\begin{align*}
\struc{V} \ = \ \left[(x_1, \dots ,x_n) \in (2\N)^n \ : \ \sum_{i=1}^{n} x_i \leq 2d \right] - \Vector{0}
\end{align*}
denote the lexicographically ordered list of all even lattice point (excluding $\Vector{0}$) with 1-norm at most $2d$.
In what follows, we represent $V$ as a $ (\binom{n+2d-1}{2d}-1) \times n$ matrix where the $j$-th row contains the $j$-th entry in the list.
Generating all $n$-simplices containing the origin thus is equivalent to listing all full-rank $n \times n$ submatrices of $V$. In order to reduce the number of necessary rank computations, we construct submatrices row by row, adding rows in the order of $V$.

We introduce some additional notation: Given a \struc{$k$-\textit{index set}} $J \subset [\#V]$, $\#J = k$, we denote its entries by $\struc{J_1},\ldots,\struc{J_k}$ and we denote by $\struc{V_J}$ the submatrix of $V$ formed by the row indices $J$. 
We call a $k$-index set $J$ a \struc{\textit{prefix}} of an $n$-index set $M$ if $k \leq n$ and $J_i = M_i$ for $i \in [k]$.

If we have an $n$-index set $I$, then the following algorithm computes the lexicographically next $n$-index set $J$ such that $\rank(V_{J})=n$.

\begin{algorithm}~

		\noindent \INPUT{$V$: Matrix of valid simplex vertices, $I$: $n$-index set}\\
		\OUTPUT{$J$: the lex-next $n$-index set with $\rank(V_J)=n$, if such a set exists; $\emptyset$: otherwise}
		\begin{algorithmic}[1]
			\State $J\gets \textrm{ lexicographic successor of }I$
			\State $K\gets \emptyset$
			\For{$j\in [n]$}
			\State $K\gets K\cup \{J_j\}$
			\If{$\textrm{rank}(V_K)< \# K$}
			\If{there exists a $\# K$-index set $\hat K$ on $[\# V - \# K]$ with $\hat K >_{\text{lex}} K$}
			\State $J\gets \min_{lex} \left\{\hat K \cup M \textrm{ $n$-index set} : \hat K>_{\text{lex}}K,\max(\hat K)<\min(M)\right\}$
			\State $K \gets \text{ largest prefix of } J \text{ contained in } K$
			\State $j \gets \# K + 1$
			\Else
			\State \Return $\emptyset$
			\EndIf
			\EndIf
			\EndFor
			\State \Return $J$
		\end{algorithmic}
		\label{algorithm:enum}
\end{algorithm}

\begin{proof}
Correctness of the algorithm is clear by construction.
In line 7, $J$ is lexicographically increased, and $K$ is always a prefix of $J$. 
The condition in line 6 does not hold for any prefix of the lexicographically maximal $n$-index set on $[\# V]$, hence the algorithm terminates.
\end{proof}

Let $K$ be an index set and $l \in [\# V]$ such that $l \notin K$. 
If $\rank(V_K) = \rank(V_{K \cup \{l\}})$, then \cref{algorithm:enum} excludes all instances containing $K \cup \{l\}$.
Therefore, we avoid the rank checks for all further matrices containing the $V_{K \cup \{l\}}$.

For the enumeration process, we use \cref{algorithm:enum} in parallel by assigning to one thread the enumeration of all matrices that have a distinguished $p \in V$ as their first row. 
%If we denote the ordered list of matrices that contain $p$ as their first row by $V_p$, then we can merge $V_p$ with respect to the order on $V$ to write down the ordered list of all full-rank $n \times n$ submatrices of $V$. 
The threads can then be run as independent processes as no inter-thread communication is required.
We did so using the \texttt{GNU parallel} software \cite{Tange:GNUParallel}.

\subsection{Classifying Simplices}
Two different simplicial sets $\Delta_1, \Delta_2 \subseteq (2\Z)^n$ may have the same maximal mediated set structure, i.e., there exists a $T \in \MMSPreservingGroup$ such that $T(\Delta_1)=\Delta_2$.
\cref{Corollary:MMSisInvariantofLattice} implies that $h$-ratio is an invariant of an underlying lattice rather than of the simplicial set itself.
Therefore, instead of the distribution of $h$-ratios over $n$-simplicial sets with maximal degree $2d$, we consider the distribution of $h$-ratios over possible lattices that can arise from $n$-simplicial sets with maximal degree $2d$.
For any $n$-simplicial set $\Delta$, we want to find a unique representative from the set $\left\{ T(\Delta) : T \in \MMSPreservingGroup \right\}$.
Unfortunately, computing the Hermite normal form straightforwardly is not sufficient to find a unique representation, because as shown in \cref{Example:MMSInvariantofLattice2}, reordering the vertices yields different Hermite normal forms.
Hence, we consider all orderings, compute their Hermite normal forms, and check whether we have already encountered that lattice before.

\begin{remark}
\label{Remark:HNFReduction}
Let $S$ be a finite set of $n$-simplicial sets. The \struc{\textit{Hermite normal form reduction (HNF reduction)}} of $S$ is the process of identifying every $n$-simplicial set in $S$ if they reduce to the same (row) Hermite normal form up to permutation of columns. 
\end{remark}

The aforementioned check is only feasible if a fast lookup of previously encountered lattices is available. 
As for the instances of interest the set of these is larger than the memory available to us, we had to resort to an on-disk key-value store.
We also required support for deadlock-free lookups and writes from multiple threads so we could still benefit from the parallelizability of the enumeration.
We ended up using a \texttt{BerkeleyDB} \cite{Olson:Bostic:Seltzer:BerkeleyDB} database, storing Hermite normal form as key and their maximal mediated set, companion matrices, $h$-ratio and other interesting information as value.

\subsection{Computing MMS}
We compute the maximal mediated set using a slightly different approach than \cref{algorithm:MMS1}.
Note that to compute $\Delta^*$, it is enough to compute $\Delta^*\cap(2\Z)^n$ since each $\Vector{\alpha} \in \Delta^* \setminus (2\Z)^n$ is midpoint of two distinct points in $\Delta^* \cap (2\Z)^n$.
In the following algorithm, we compute $\Delta^* \cap (2\Z)^n$ by starting with an lex-ordered list $L$ of all points in $\conv(\Delta) \cap (2\Z)^n$ and iteratively removing all points that are not midpoints of two points in $\conv(\Delta) \cap (2\Z)^n$. 
The notions $\struc{\ListHead(L)}$ and  $\struc{\ListTail(L)}$ denote the first and the last element of the lex-ordered list $L$.
\begin{algorithm}~
% % % 	Given a finite $\Delta \subseteq (2\Z)^n$, the following algorithm computes the maximal $\Delta$-mediated set, $\Delta^*$. \\

	\noindent \INPUT{$\Delta$: finite set of affine independent points in $(2\Z)^n$}\\ 
	\OUTPUT{$\Delta^*$: the $\Delta$-mediated subset of $\Z^n$ that contains every $\Delta$-mediated set}
	\begin{algorithmic}[1]
		\State $L\gets[(\conv(\Delta)\cap (2\Z)^n]$
		\State $i\gets \ListTail(L)$
		\While{$i\neq \ListHead(L)$}
		\If{$i\notin\MidD(L)\cup\Delta$}
		\State $L \gets L \setminus \{i\}$
		\State $i \gets \ListTail(L)$
		\Else
		\State decrement $i$
		\EndIf
		\EndWhile\\
		\Return $L\cup \MidD(L)$
	\end{algorithmic}
	\label{algorithm:MMS2}
\end{algorithm}

% % % 
% % % 
% % % Next we show that \cref{algorithm:MMS2} indeed computes the $\Delta^*$.
% % % \begin{proposition}
% % % 	Given an affine independent set $\Delta \subset (2\Z)^n$, the outputs of \cref{algorithm:MMS1} and \cref{algorithm:MMS2} are the same.
% % % \end{proposition} 

\begin{proof}[Correctness and Termination]
First we prove that the algorithm terminates, i.e., the while loop in the algorithm terminates.
If $L$ contains only one element, then the while loop immediately terminates since the condition in line 3 is satisfied.
Assume $L$ contains more than one element.
\begin{enumerate}
	\item If the condition in line 4 is not satisfied for any $i$ as $i$ runs through $L$, then the while loop terminates.
	\item If the condition in line 4 is satisfied for some $i$ as $i$ runs through $L$, then $i$ is removed from $L$ and while loop restarts.
\end{enumerate}
Since $L$ has a finite cardinality $k$, the while loop terminates after at most $k-1$ restarts.

For the correctness of the algorithm, let $\Delta_0^*$ denote the output of \cref{algorithm:MMS2} with the input $\Delta$. 
We show that $\Delta^* \subset \Delta_0^*$ and $\Delta_0^*$ is $\Delta$-mediated, thus it follows that $\Delta^* = \Delta_0^*$ by maximality of $\Delta^*$.

For the first one it is enough to show the claim for even points, i.e. $\Delta^* \cap (2\Z)^n \subset \Delta_0^* \cap (2\Z)^n $.
In order to argue by contradiction, we assume that $ D = (\Delta^* \cap (2\Z)^n) \setminus  (\Delta_0^* \cap (2\Z)^n)$ is non empty.
When algorithm is initialized, the list $L$ is set to $\conv(\Delta) \cap (2\Z)^n$ which contains $D$.
As the algorithm runs through, some elements of $D$ are discarded one by one. 
Let $\Vector{\alpha}$ denote the first element of $D$ that was discarded from $L$. 
Note that $D \cap \Delta = \emptyset$ because $\Delta$ is a subset of both $\Delta^*$ and $\Delta_0^*$.
Let $L_{\Vector{\alpha}}$ denote the elements that stay in the list $L$ until $\Vector{\alpha}$ is discarded. 
Since $\Vector{\alpha} \in \Delta^* \setminus \Delta $, there exists distinct $\Vector{\alpha_1}, \Vector{\alpha_2} \in \Delta^*$ such that $\Vector{\alpha} = \frac{\Vector{\alpha_1} + \Vector{\alpha_2}}{2}$. 
Note that by definition of $\Vector{\alpha}$, both $\Vector{\alpha_1}$ and $\Vector{\alpha_2}$ are in $L_{\Vector{\alpha}}$.
Because, otherwise $\Vector{\alpha_1}$ or $\Vector{\alpha_2}$ would be the first element in $D$ to be removed from $L$.
Therefore, we have that $\Vector{\alpha} \in \MidD(L_\alpha)$ and condition on step 4 fails to hold.
However, this implies that $\Vector{\alpha}$ is not discarded which contradicts to the fact that $\Vector{\alpha} \in D$.
Therefore, $D$ is empty and $(\Delta^* \cap (2\Z)^n) \subset  (\Delta_0^* \cap (2\Z)^n)$.

Lastly in order to prove $\Delta_0^*$ is $\Delta$-mediated, we let $\Vector{\beta} \in \Delta_0^* \setminus \Delta$ and show that $\Vector{\beta} \in \MidD(\Delta_0^*)$. 
Due to step 4, $\Vector{\beta} \in \Delta_0^*$ only if $\Vector{\beta} \in \MidD(\Delta^0 \cap (2\Z)^n)$.
The claim follows since $\MidD(\Delta_0^*) = \MidD(\Delta_0^* \cap (2\Z)^n)$.
\end{proof}

There are a two additional aspects of our implementation algorithm that we did not highlight in the pseudo code.
\begin{enumerate}
	\item To increase efficiency, we incorporated cost efficient pre-computation checks based on \cite[Theorem 2.5, Theorem 2.7]{Reznick:AGI}.
They let us detect $H$-simplices and $M$-simplices without going through the iteration process in some cases.
	\item We keep $L$ in lexicographical order, which ensures that for any point $i\in L$, if it is midpoint of two points in $L$ then one of those two will appear before $i$ in the list and  will appear after $i$.
This enables us to find out whether $i\in \MidD(L)$ in at most $\frac{1}{4}(\#L)^2$ operations.
\end{enumerate}

\section{Computational Results and Statistics}
In this section, we analyze the experimental results we achieved. 
Our main measurement is the $h$-ratio defined in \cref{Definition:hratio}, and we are interested how the $h$-ratio is distributed for fixed dimension and degree.
\subsection{Experimental Setup}
\label{Subsection:ExperimentalSetup}
Here we give the overview of the experimental setup we have used.
\begin{description}
	
	\item[Software] We performed the maximal mediated set computations in the open source software \textsc{Polymake}.
	We have written our own extension to \textsc{Polymake} that computes the maximal mediated set via \cref{algorithm:MMS2}.
	The package is available at:
	\begin{center}
		\href{https://polymake.org/doku.php/extensions/max\_mediated\_sets}{https://polymake.org/doku.php/extensions/max\_mediated\_sets}
	\end{center}
	
	\item[Investigated Data] The investigated data consists of simplicial sets and the lattices underlying simplicial sets described in \cref{Corollary:MMSisInvariantofLattice}.
	We divide the investigated data into smaller sets according to two parameters:
	\begin{description}
		\item[n] the dimension of the simplicial sets. 
		\item[2d] the maximal total degree of the simplicial sets, i.e. the minimum integer for a given simplicial set $\Delta$ such that for all $\Vector{p} \in \conv(\Delta)$, $\Vector{1}^T \cdot \Vector{p} < 2d$.
	\end{description}
	\begin{table}[h]
		\centering
		\begin{tabular}{c|cccccc}
			dimension & 2 & 3 & 4 & 5 & 7 &  \\
			\hline
			degree    & 150 & 16 & 14 & 8 & 4 & 
		\end{tabular}
		\caption{\small Maximal degrees and dimensions of the fully computed data.}
		\label{table:datainfo_full}
	\end{table}
	\cref{table:datainfo_full} summarizes for which $n,2d \in \N$ we have computed MMS of every possible instance.
	We resort to sampling instead of computing the maximal mediated set for all simplicial sets for larger $n$ and $2d$.
	\cref{table:datainfo_sampled} summarizes the cases for which values of $n,2d \in \N$ we have sampled $n$-simplicial sets $\Delta \subset (2\Z)^n$ such that zero vector is a vertex of $\conv(\Delta)$. The database is available via
	\begin{center}
		\href{https://polymake.org/downloads/MMS/}{https://polymake.org/downloads/MMS/}
	\end{center} 
	\begin{table}[t]
		\centering
		\begin{tabular}{c|cccccc}
			dimension  & 4 & 5 & 6 & 7 & 8 & 9  \\
			\hline
			degree     & 16  & 16 & 16 & 16 & 16 & 16
		\end{tabular}
		\caption{\small Maximal degrees and dimensions of the sampled data.}
		\label{table:datainfo_sampled}
	\end{table}
	
	\item[Sampling] For reproducibility of our sampling, we provide the \texttt{RandomSimplexIterator} class in our \textsc{Polymake} package, using the type \texttt{UniformlyRandom<Integer>} in \textsc{Polymake}. This class initially produces an array of integers from a seeded uniform distribution with the seed 1. 
	Then it picks $n$ points in $\N^n$ with 1-norm at most $2d$ uniformly at random seeded by the elements the integer array.
	If these $n$ points together with the origin are affine independent then we keep this sample, otherwise we discard it and pick another $n$ point uniformly at random.
	We did not set a specific sample size as stopping criterion, since we aim to compute as much MMS as we can.
	Thus, we stopped the sampling process according to elapsed time for each $n$ and $2d$, see the point Sampled Data Sets in \Cref{Subsection:EvaluationOfTheExperiment}	
	
	\item[Hardware and System] We used three separate computers for the computations. 
	For $n = 4$ and $2d = 14$, we used a \verb+AMD Phenom(TM) II X6 1090T+ with 5 cores, 16 GB of RAM under openSUSE Leap 15.0.
	For $n = 5$ and $2d = 8$, we used a \verb+AMD Phenom(TM) II X6 1090T+ with 6 cores, 16 GB of RAM under openSUSE Leap 42.3.
	For the remaining computations, we used \verb+Intel(R) Core(TM) i7-8700+ \verb+CPU @ 3.20GHz+ with 12 cores, 16 GB of RAM under openSUSE Leap 15.0.
	
%	\item[Runtime and Memory] The run time of computation for $n = 4$ and $d = 15$ was 21 days and it created 1.8 GB of data.
%	The run time of computation for $n = 7$ and $d = 4$ was 18 hours and it created 17MB of data.
%	Lastly, the run time for remaining cases was 

\end{description}

\subsection{Evaluation of the Experiment}
\label{Subsection:EvaluationOfTheExperiment}
In this section we present and evaluate the results of our computation.

\begin{description}
	\item[Maximal Mediated Subsets of 2-Simplicial Sets] First we address \cref{Conjecture:ReznickDim2}, which was announced in \cite[Page 9]{Reznick:AGI}. 
	We computed the maximal mediated sets of all 4266834 2-simplicial sets with maximal degree 150, and confirmed that \cref{Conjecture:ReznickDim2} holds.
	These 2-simplicial sets arise from 886297 different lattices as described in \cref{Corollary:MMSisInvariantofLattice} after an Hermite normal form reduction.
	We summarize the statistics corresponding to $n = 2$, $2d = 150$ case in \cref{table:dim2deg150}.
	\begin{table}[h]
		\centering
		\begin{tabular}{c|ccccc}			
			& Total count & $H$-simplex & $M$-simplex  & mean of $h$-ratio & SD of $h$-ratio  \\
			\hline
			2-Simplicial Sets & 4266834 & 4250533  & 16301 & 0.996179 & 0.061691 \\
			Lattices & 886297 & 886188 & 109 & 0.999877 & 0.011089\\
		\end{tabular}
		\caption{\small From left to right, the total number of simplices, the number of $H$-simplices, the number of $M$-simplices, the average $h$-ratio, and the standard deviation of the $h$-ratio for 2-simplicial sets and their underlying lattices in the case $n = 2$, $2d = 150$.}
		\label{table:dim2deg150}
	\end{table}
	From the \cref{table:dim2deg150}, we see that number of $H$-simplices is significantly higher than number of $M$-simplices; a fact suggested by \cite[Theorem 5.9]{Iliman:deWolff:Circuits}.
	Thus, we provide experimental evidence that a clear majority of all non-negative circuit polynomials in $\R[x,y]$ are sums of squares.	
	\item[Fully Computed Data Sets for Higher Dimensions] 
	Here we present the fully computed data and show that in these cases the Hermite normal form reduction indeed yields a different distribution for $h$-ratio. 
	First, we summarize the general statistics we obtained from the fully computed data sets in \cref{table:fullycomputedstats}. 
	\begin{table}[h]
		\centering
		\begin{tabular}{cc|c|ccc}
			
			n                  & 2d					 &                & Total Count & Mean of $h$-ratio & SD of $h$-ratio \\
			\hline
			\multirow{2}{*}{3} & \multirow{2}{*}{10} & Simplicial Set & 21636       & 0.724138        & 0.392967 \\ 
			 
			&                        				& Lattice        & 782          & 0.592994 & 0.397988 \\

			\multirow{2}{*}{3} & \multirow{2}{*}{16} & Simplicial Set & 659082        & 0.638828        & 0.412316 \\ 
			
			&                        				& Lattice        & 20429          & 0.583357 & 0.412889 \\ 
			
			\multirow{2}{*}{4} & \multirow{2}{*}{14} & Simplicial Set & 853024289       & 0.433506        & 0.383378 \\ 
			
			&                        				& Lattice        & 1602368          & 0.227706 	 & 0.273419 \\ 
			
			\multirow{2}{*}{5} & \multirow{2}{*}{8} & Simplicial Set & 305565979       & 0.680445        & 0.373089 \\ 
			
			&                        				& Lattice        & 53306          & 0.470493 	 &  0.303315 \\ 
			
			\multirow{2}{*}{7} & \multirow{2}{*}{4} & Simplicial Set & 2414505       & 0.931788        & 0.238172 \\ 
			
			&                        				& Lattice        & 19          & 0.853923	 & 0.304942 \\ 
			
		\end{tabular}
		\caption{\small This table summarizes the statistics of $n$-simplicial sets with maximal degree $2d$ and the statistics of lattices underlying $n$-simplicial sets with maximal degree $2d$ for fully computed data sets.}
		\label{table:fullycomputedstats}
	\end{table}
	
	We observe that the number of underlying lattices is significantly smaller than the number of simplicial sets. 
	To point out this difference more rigorously, we provide in \cref{table:fullycomputeddecreasefactor} the factors of decrease for each $n$ and $2d$ in the fully computed cases.
	
	\begin{table}[t]
		\centering
		\begin{tabular}{cc|c}
			n                  & 2d					 & Decrease Factor \\
			\hline
			3 & 10 & 27.667519    \\ 
			3 & 16 & 32.262078 \\
			4 & 14 & 532.352299 \\
			5 & 8  & 5732.299910 \\
			7 & 4 & 127079.210526 \\
			
		\end{tabular}
		\caption{\small The factors of decrease in the number of stored maximal mediated sets after a Hermite normal form reduction is performed.}
		\label{table:fullycomputeddecreasefactor}
	\end{table}	

	\cref{table:fullycomputeddecreasefactor} reveals that the decrease factor is more sensitive to a change in $n$ than a change in $2d$.
	We explain this experimental observation theoretically in what follows. Let $\Delta$ be a $n$-simplicial set and consider the coordinate permutation given by $\Pi_{\sigma}: x_i \mapsto x_{\sigma(i)}$ where $\sigma \in S_n$. Observe that:
	\begin{enumerate}
		\item Since $\Pi_{\sigma} \in \MMSPreservingGroup(\R^n)$, $\Delta$ and $\Pi_{\sigma}(\Delta)$ share the same Hermite normal form after a suitable permutation of columns. Therefore, $\Delta$ and $\Pi_{\sigma}(\Delta)$ are in the same class of lattices in our database. 
		\item Maximal degrees of $\Delta$ and $\Pi_{\sigma}(\Delta)$ are equal.
		\item Increasing the dimension from $n$ to $n+1$ yields $(n+1)! - n!$ new possible coordinate permutations. 
	\end{enumerate}
	Therefore, increasing $n$ increases the number of simplicial sets with the same Hermite normal form exponentially even if maximal degree $2d$ is fixed.
	However, increasing $2d$ while $n$ is fixed does not create new symmetries and hence does not affect the decrease factor as severely.
	The decrease in the total count of objects in \cref{table:fullycomputedstats} is crucial to reduce the size of the database.
	However, the HNF reduction is computationally costly since making use of \cref{Corollary:MMSisInvariantofLattice} requires to consider all column permutations of the underlying matrix.
	Thus, there is a trade off between the memory, which is needed for the database, and the time needed to compute the database.
	
	Recall that MMS is an invariant of the underlying lattice of its defining simplex by \cref{Corollary:MMSisInvariantofLattice}. Therefore, considering lattices (instead of the original simplicial sets) yields a more accurate description of the behavior of $h$-ratio.
	We plot the distributions of simplicial sets and lattices in \cref{Figure:FullComputedSummary} to visualize the significant difference between the distributions.
	\begin{figure}[t]
		\centering
		\ifpictures
		\includegraphics[width=0.9\linewidth]{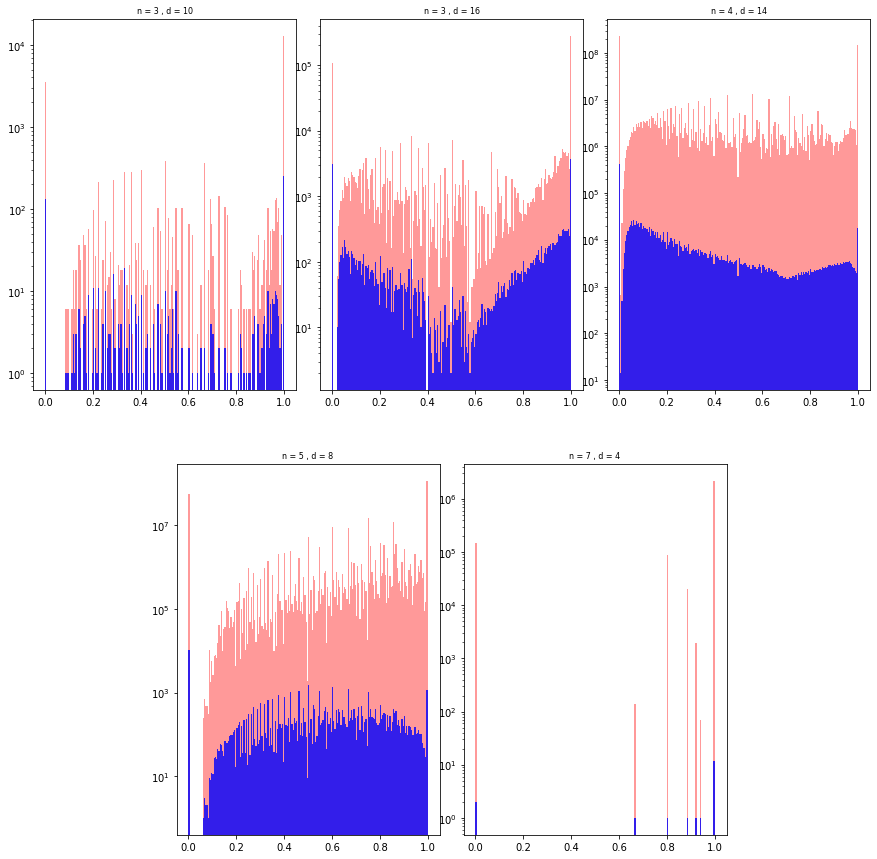}
		\fi
		\caption{\small The distribution of $h$-ratio over the $n$-simplicial sets (red), and over lattices (blue) for the given $n$ and $2d$ in \cref{table:datainfo_full}.}
		\label{Figure:FullComputedSummary}
	\end{figure}	
	The difference of the distributions follows moreover from our results in \cref{table:fullycomputedstats} and \cref{table:fullycomputeddecreasefactor}.
	More specifically, we observe that Hermite normal form reduction decreases the expected $h$-ratio.
	The standard deviation for simplicial sets is more stable compared to standard deviation for lattices.

	\item[Sampled Data Sets] For higher $n$ and $2d$ we produced a sampled database as described in the sampling part of \cref{Subsection:ExperimentalSetup}. In \cref{table:sampledstats} we present the summary of $h$-ratio statistics of the sampled cases for simplicial sets and underlying lattices and in \cref{Figure:SampledSummary} we plot the distributions of simplicial sets and underlying lattices.
	
		\begin{table}[t]
		\centering
		\begin{tabular}{cc|c|ccc}
			
			n                  & 2d					 &                & Total Count & Mean of $h$-ratio & SD of $h$-ratio\\
			\hline
			\multirow{2}{*}{4} & \multirow{2}{*}{16} & Simplicial Set & 10000000       & 0.392896        & 0.370466 \\ 
			
			&                        				& Lattice        & 2067884       & 0.221803        & 0.277060 \\

			\multirow{2}{*}{5} & \multirow{2}{*}{16} & Simplicial Set & 5000000        & 0.299490        & 0.320094 \\ 
			
			&                        				& Lattice        & 3297468          & 0.216518 & 0.245109 \\ 
			
			\multirow{2}{*}{6} & \multirow{2}{*}{16} & Simplicial Set & 1000000       & 0.290170        & 0.322581 \\ 
			
			&                        				& Lattice        & 904317          & 0.263667 	 & 0.297612 \\ 
			
			\multirow{2}{*}{7} & \multirow{2}{*}{16} & Simplicial Set & 100000       &  	0.325715        & 0.361047 \\ 
			
			&                        				& Lattice        & 98016          & 0.319911 	 &  0.356507 \\ 
			
			\multirow{2}{*}{8} & \multirow{2}{*}{16} & Simplicial Set & 10000       & 0.387188        & 0.411047 \\ 
			
			&                        				& Lattice        & 9966          & 0.385937	 & 0.410454  \\ 
			
			\multirow{2}{*}{9} & \multirow{2}{*}{16} & Simplicial Set & 1302       & 0.625456        & 0.447447 \\ 
			
			&                        				& Lattice        & 1000          & 0.512343	 & 0.453334 \\ 
		\end{tabular}
		\caption{\small This table summarizes the statistics of $n$-simplicial sets with maximal degree $2d$ and the statistics of lattices underlying $n$-simplicial sets with maximal degree $d$ for fully computed data sets.}
		\label{table:sampledstats}
	\end{table}
	
	\begin{figure}[t]
		\centering
		\ifpictures
		\includegraphics[width=0.9\linewidth]{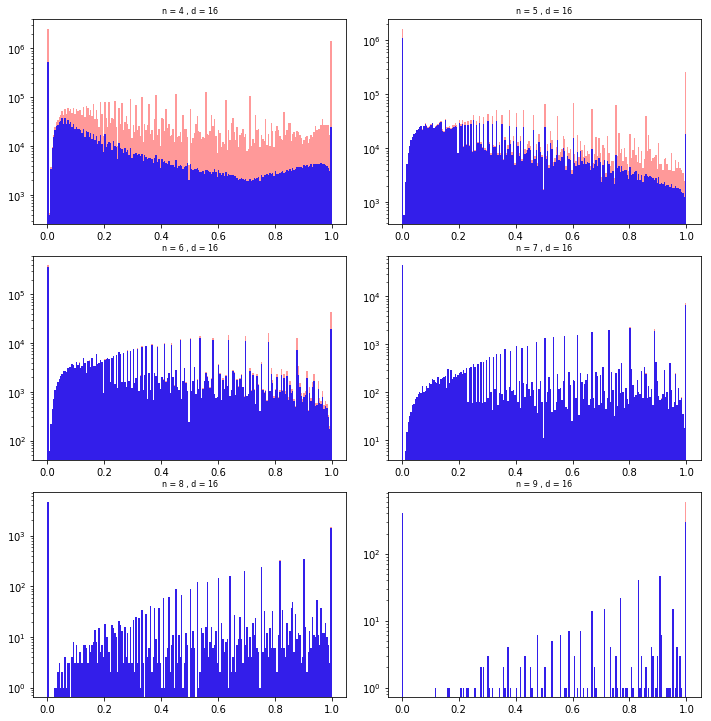}
		\fi
		\caption{\small The sampled distributions of $h$-ratio over the $n$-simplicial sets (red), and over lattices (blue) for the given $n$ and $2d$ in \cref{table:datainfo_sampled}.	}
		\label{Figure:SampledSummary}
	\end{figure}
	
	We note that the sampled data for $ n = 8 $ and $ n = 9 $ are incomplete, and hence do not provide much information about the statistics.
	However we include them in \cref{table:sampledstats} and \cref{Figure:SampledSummary}, to underline the fact that our sampling approach does not operate smoothly with HNF reduction for higher values of $n$.
	This is due to the exponential increase in the column permutations that one has to check for HNF reduction.	
	
% % % 	\timo{We have to delete this paragraph up to ``Hermite Normal Form'' (or completely reformulate it). It is way too negative. Also, I do not see why it is important.}
% % % 	We observe that considering lattices does not yield a significant decrease factor in the size of the stored data in contrast to the fully computed cases.
% % % 	On one hand, this fact undermines the idea of utilizing \cref{Corollary:MMSisInvariantofLattice} to reduce the size of database.
% % % 	On the other hand, this fact provides further insight on the effect of the Hermite normal form reduction that we exploit.
% % % 	Specifically, this means that our sampling method is less likely to produce simplicial sets that does not reduce to same Hermite normal form.

	While the sampled data does not yield a clear, noise-free distribution, we still observe a similar trend in the expected $h$-ratios for $n \in \{4,5,6,7\}$. Analog to the fully computed (i.e., non-sampled) cases, the expected $h$-ratio of lattices are smaller than expected $h$-ratio of simplices also in the sampled cases.

	\item[Further Effects of HNF Reduction to Statistics] We observe some peaks in the graphs provided in \cref{Figure:FullComputedSummary}.
	In order to study these peaks closer, we plot the $h$-ratio distributions of the case $n=4$ as $2d$ increases in \cref{Figure:Dim4Progress}.
	\begin{figure}[t]
	\centering
	\ifpictures
	\includegraphics[width=0.9\linewidth]{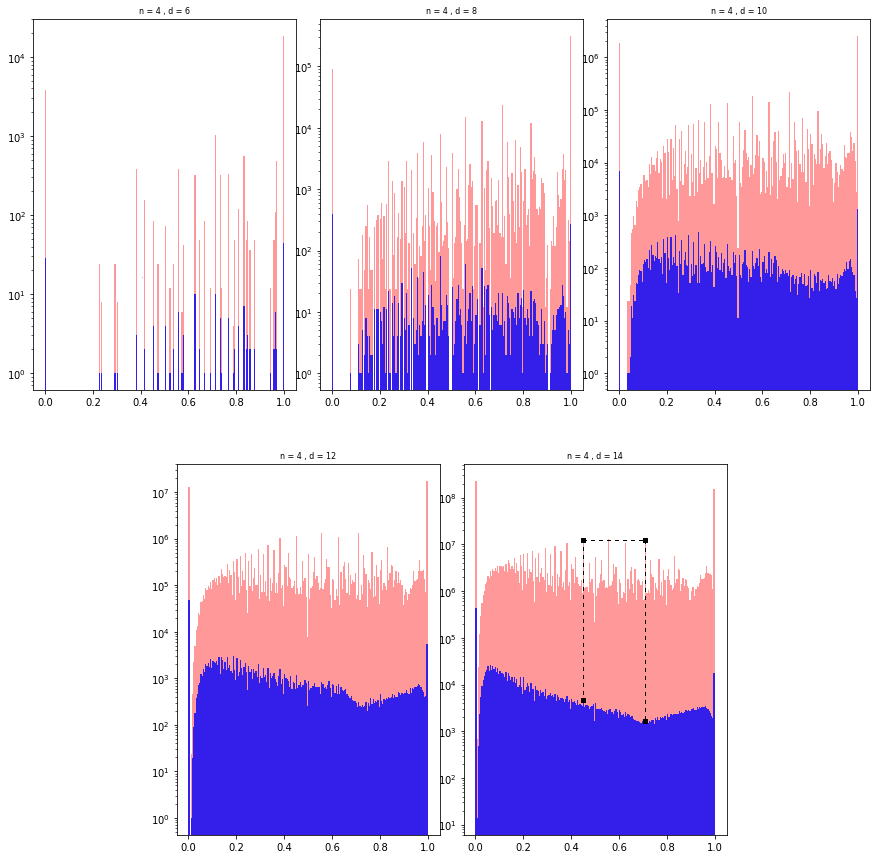}
	\fi
	\caption{\small The distributions of $h$-ratio over the $4$-simplicial sets (red), and over lattices (blue) for $2d=6,8,10,12,14$. On bottom right, we see that two spikes of same height are effected differently from HNF reduction.	}
	\label{Figure:Dim4Progress}
	\end{figure}
	For small values of $2d$ we have individual peaks because small maximal degree does only allow a few different  $h$-ratios to occur. 
	For sufficiently large $2d$ the distribution becomes visible, since larger variety of $h$-ratios can appear. 
	Furthermore, as $2d$ increases, the individual peaks that exists for smaller $2d$ survive and form the spikes we observe in the red distributions.
	\cref{Figure:Dim4Progress} illustrates how as $2d$ increases, individual peaks transform into a noisy distribution with spikes.

	Even though the spikes present in both distribution of $h$-ratio both over simplicial set (red) and over lattices (blue), they are visually more evident red distributions.
	By Hermite normal form reduction, we shrink down the sizes of the bins in the graphs. Therefore, observing smaller spikes in the blue graphs is expected.
	We observe from lower right graph in \cref{Figure:Dim4Progress}, that the shrinking of the spikes is not uniform.
	This is an expected observation since we already know that Hermite normal form reduction changes the expected $h$-ratios of the distributions.
	In addition to this, we see that shrinking of the spikes induced by the Hermite normal form reduction smoothens the distribution of the $h$-ratio.
	Therefore, considering the $h$-ratio distribution over lattices is not only plausible mathematically but also statistically.
	We have already mentioned for both sampled and fully computed cases, that the mean of the $h$-ratio distribution over the lattices is less than the mean over the simplicial sets.
	This observation suggests that Hermite normal form reduction has an non-trivial effect on $h$-ratio distribution. 
	Furthermore, interpreting this observation in terms of circuit polynomials, we conclude that circuit polynomials are less likely to be SOS with respect to the more valid statistics where we consider lattices instead of simplicial sets.

\end{description}
 
\section{Resume}
	
	In conclusion, we provide two main contributions. First, we show that the maximal mediated set structure, in particular the $h$-ratio, of a simplicial set is defined by the underlying lattice described in \cref{Corollary:MMSisInvariantofLattice}. 
	
% % % 	On the other hand, the computational cost of creating the database increases exponentially with the dimension since one has to calculate Hermite Normal Form of every column permutation the lattice $L_{\Delta}$ for each $\Delta$ to use \cref{Corollary:MMSisInvariantofLattice}. 
	
	Second, we provide a large database of MMS that is partitioned according to the number of variables, $n$, maximal degree of the circuit polynomials $2d$ for $n,2d$ given in \cref{table:datainfo_full} and a sampled database with a similar partitioning for $n$ and $2d$ given in \cref{table:datainfo_sampled}.
	
	\begin{center}
		\href{https://polymake.org/downloads/MMS/}{https://polymake.org/downloads/MMS/}
	\end{center} 
	
	\cref{Corollary:MMSisInvariantofLattice} allows us to reduce the size of the database substantially and smoothens the spikes we observe in the data in exchange of an increased runtime.	
	Furthermore, we present indications concerning the distribution of the $h$-ratio.
	We observe that Hermite normal form reduction has a non-trivial effect on the $h$-ratio distribution, since it changes the mean of the distribution. 
	Moreover, for $n = 2$ up to $2d = 150$ we show that:
	\begin{enumerate}
		\item the distribution of $h$-ratio is a Bernoulli distribution.
		\item expected value is close to 1.
	\end{enumerate}
	The first part computationally proves \cref{Conjecture:ReznickDim2} up to $2d=150$. Yet, the full proof of the conjecture is still missing. The second observation follows then from (1) and \cite[Theorem 5.9]{Iliman:deWolff:Circuits}.

	Based on the computed data, we conjecture that for $n > 2$ as $2d$ grows the distribution of $h$-ratio over lattices does \textbf{not} converge to the uniform distribution. 
	Observing the data, we believe that the distribution of the $h$-ratio for fixed $n$ as $2d$ grows might be related to the Chi-squared distribution, but we do not have hard evidence for this fact.
	% % % 	However, the true behavior distribution , is still not clear to us.
	The main issue is improving the database and accessing to the $h$-ratio distribution for higher $n$ and $2d$. 
	A possible way to discover the distribution for higher $n$ and $2d$ is to sample the lattices with uniform standard distribution, and we are not aware how this can be done. More importantly, we need a faster way to determine whether $L_{\Delta_1}$ and $L_{\Delta_2}$ share the same Hermite Normal Form up to a permutation of columns. This will not only enable us to compute the exact distribution for more cases, but also increase the speed of our sampling. 
	In addition to improving the database, we will utilize it to tackle maximal mediated sets with methods from combinatorial algebraic geometry in a future project.
	
\bibliographystyle{amsalpha}
\bibliography{MMS_Computation}

\end{document}